\documentclass[12pt,a4paper]{amsart}
\usepackage{amsfonts,color}
\usepackage{amsthm}
\usepackage{amsmath,amssymb}
\usepackage{amscd}
\usepackage[latin2]{inputenc}
\usepackage{t1enc}
\usepackage[mathscr]{eucal}
\usepackage{indentfirst}
\usepackage{graphicx}
\usepackage{graphics}
\usepackage{pict2e}
\usepackage{epic}
\usepackage{url}
\usepackage{epstopdf}

\numberwithin{equation}{section}
\usepackage[margin=2.6cm]{geometry}

 \usepackage{tikz}
\usetikzlibrary{shapes}

\tikzstyle{vertex}=[draw=black,circle,fill=black,minimum size=4pt, inner sep=0pt, outer sep=0pt,text=white,line width=0mm]
\tikzstyle{vertex_blue}=[draw=black,circle,fill=blue,minimum size=4pt, inner sep=0pt, outer sep=0pt,text=white,line width=0mm]
\tikzstyle{vertex_red}=[draw=black,circle,fill=red,minimum size=4pt, inner sep=0pt, outer sep=0pt,text=white,line width=0mm]
\tikzstyle{vertex_green}=[draw=black,circle,fill=green,minimum size=4pt, inner sep=0pt, outer sep=0pt,text=white,line width=0mm]
\tikzstyle{vertex_yellow}=[draw=black,circle,fill=yellow,minimum size=4pt, inner sep=0pt, outer sep=0pt,text=white,line width=0mm]
\tikzstyle{c0}=[shape=circle, minimum size=4pt, fill=white]
\tikzstyle{c1}=[shape=rectangle, minimum size=7pt, fill=red]
\tikzstyle{c2}=[shape=diamond, minimum size=10pt, fill=blue]
\tikzstyle{mybox} = [rectangle, rounded corners, minimum width=3cm, minimum height=1cm,text centered, draw=black]
\tikzset{base/.style = {rectangle, rounded corners, draw=black,
                           minimum width=3cm, minimum height=1cm,
                           text centered}}

\newcommand\scalemath[2]{\scalebox{#1}{\mbox{\ensuremath{\displaystyle #2}}}}

\theoremstyle{plain}
\newtheorem{Th}{Theorem}[section]
\newtheorem{Lemma}[Th]{Lemma}
\newtheorem{Cor}[Th]{Corollary}
\newtheorem{Prop}[Th]{Proposition}
\newtheorem{observation}[Th]{Observation}

 \theoremstyle{definition}
\newtheorem{Def}[Th]{Definition}

\newtheorem{Rem}[Th]{Remark}
\newtheorem{?}[Th]{Problem}
\newtheorem{Ex}[Th]{Example}

\newcommand{\A}{\textbf{A}}
\newcommand{\B}{\textbf{B}}
\newcommand{\R}{\mathbb{R}}
\newcommand{\E}{\mathbb{E}}

\begin{document}

\title[Determinental Sidorenko's conjecture and GMRF]{On Sidorenko's conjecture for determinants and Gaussian Markov random fields}

\author[P. Csikv\'ari]{P\'{e}ter Csikv\'{a}ri}

\address{Alfr\'ed R\'enyi Institute of Mathematics \\ H-1053 Budapest, Hungary \\ Re\'altanoda utca 13-15 \and E\"otv\"os Lor\'and University, Institute of Mathematics, Department of Computer Science, H-1117 Budapest, Hungary\\ P\'azm\'any P\'eter  s\'et\'any 1/C} 

\email{peter.csikvari@gmail.com}

\author[B. Szegedy]{Bal\'azs Szegedy}

\address{Alfr\'ed R\'enyi Institute of Mathematics \\ H-1053 Budapest \\ Re\'altanoda utca 13-15}

\email{szegedyb@gmail.com}

\thanks{The first author is  supported by the Counting in Sparse Graphs Lend\"ulet Research Group of the Alfr\'ed R\'enyi Institute of Mathematics.  When the project started he was also  supported by the 
 Marie Sk\l{}odowska-Curie Individual Fellowship grant no. 747430. The second author has received funding from the European Research Council under the European Union's Seventh Framework Programme (FP7/2007-2013) / ERC grant agreement n$^{\circ}$617747. The research was partially supported by the MTA R\'enyi Institute Lend\"ulet Limits of Structures Research Group.
}



\begin{abstract} We study a class of determinant inequalities that are closely related to Sidorenko's famous conjecture (also conjectured by Erd\H os and Simonovits in a different form). Our main result can also be interpreted as an entropy inequality for Gaussian Markov random fields (GMRF). We call a GMRF on a finite graph $G$ homogeneous if the marginal distributions on the edges are all identical. We show that if $G$ is bipartite, then the differential entropy of any homogeneous GMRF on $G$ is at least $|E(G)|$ times the edge entropy plus $|V(G)|-2|E(G)|$ times the point entropy. We also show that in the case of non-negative correlation on edges, the result holds for an arbitrary graph $G$. The connection between Sidorenko's conjecture and GMRF's is established via a large deviation principle on high dimensional spheres combined with graph limit theory. It is also observed that the system we study exhibits a phase transition on large girth regular graphs. Connection with Ihara zeta function and the number of spanning trees is also discussed.
\end{abstract}

\maketitle

\section{Introduction}

Gaussian Markov random fields (GMRF's) are fundamental constructions in various areas of mathematics including statistics, computer science, machine learning, statistical physics and probability theory \cite{CheCha,CFP,DLP,Mar,RueHeld}. Continuous versions include Gaussian free fields (GFF's) that are extensively used in quantum field theory. Despite the fact that GMRF's are defined on finite graphs, their study from a graph theoretic point of view is less prevalent.  The main goal of this paper is to initiate a line of research that focuses on the interplay between the properties of a graph $G$ and the possible GMRF's that can be realized on $G$. Inside this larger framework we pick an interesting problem that is closely connected to graph limits and the famous Erd\H os-Simonovits-Sidorenko conjecture. We study GMRF's on graphs with the homogeneity property that their marginal distributions on the edges of the underlying graph are all identical. We show that entropies of such fields are, in a certain sense, limits of homomorphism densities known from graph limit theory. This enables us to study various correspondences between extremal graph theory and homogeneous GMRF's. In particular we prove the Erd\H os-Simonovits-Sidorenko conjecture in the GMRF framework.
\bigskip

\textbf{Markov random fields.} 
Markov random fields on graphs are natural generalizations of Markov chains. They
are defined as joint distributions of random variables associated with the vertices of a
given graph satisfying certain conditional independence properties. Classical Markov chains are
equivalent to Markov random fields defined on paths of finite or infinite length. Gaussian
Markov random fields (GMRF's) are special Markov random fields with the additional
property that the joint distribution of the variables is Gaussian.

More formally, given a graph $G=(V,E)$ a collection of random variables $(X_v)_{v\in V(G)}$ has the spatial Markov property if the following condition holds true: for any set $S$ the joint distribution of $(X_v)_{v\in S}$ conditioned on $(X_u)_{u\in V\setminus S}$ is the same as  the joint distribution of $(X_v)_{v\in S}$ conditioned on $(X_u)_{u\in N_G(S)}$, where $N_G(S)$ is the set of neighbors of $S$, that is, $N_G(S)=\{u \in V\setminus S\ |\ \exists v\in S: (u,v)\in E(G)\}$. We say that $(X_v)_{v\in V}$ is a Markov random field if $(X_v)_{v\in V}$ has the spatial Markov property.

\begin{figure}[h!]
    \centering
    \begin{tikzpicture}
    \node[vertex_green] (u1) at (1,0) {};
	  \node[vertex_blue] (u2) at (3,-1) {};
		\node[vertex] (u3) at (5,-1) {};
    \node[vertex] (u4) at (7,0) {};
    \node[vertex] (u5) at (8,1.25) {};
    \node[vertex] (u6) at (7,2.5) {};
    \node[vertex] (u7) at (5,3.5) {};
    \node[vertex_red] (u8) at (3,3.5) {};
		\node[vertex_yellow] (u9) at (1,2.5) {};
		\node[vertex_blue] (u10) at (0,1.25) {};
		\draw[rounded corners] (-0.75,-0.5) rectangle (1.75,3) node[anchor=north] at (0.5,-0.5) {$S$};
		\draw[rounded corners] (2.25,-1.5)  rectangle (3.75,4) node[anchor=north] at (3,-1.5) {$N_G(S)$};

  \draw (u1) -- (u2);
	\draw (u2) -- (u3);
	\draw (u3) -- (u4);
	\draw (u4) -- (u5);
	\draw (u5) -- (u6);
	\draw (u6) -- (u7);
	\draw (u7) -- (u8);
	\draw (u8) -- (u9);
	\draw (u9) -- (u10);
	\draw (u10) -- (u1);
	\draw (u1) -- (u9);
	\draw (u1) -- (u8);
	\draw (u2) -- (u8);
	\draw (u3) -- (u7);
	\draw (u3) -- (u8);
	\draw (u3) -- (u6);
	\draw (u4) -- (u6);

    \end{tikzpicture}
    \caption{A graph with a partial proper coloring. For a uniformly chosen proper coloring the distribution of the colors of the set $S$ only depends on the coloring of the set $N_G(S)$, but not on the rest of the graph.}
    \label{fig:my_label}
\end{figure}

Markov random fields naturally show up in the study of spin systems. For instance, if we consider a uniform random proper coloring of a graph $G$ with $q$ colors, and $X_v\in \{1,2,\dots ,q\}$ is the color of the vertex $v$, then it is easy to see that $(X_v)_{v\in V}$ is a Markov random field. Similarly, if we consider a uniform random independent set of $G$, and $X_v\in \{0,1\}$ describes whether $v$ is in the independent set or not, then again $(X_v)_{v\in V}$ forms a Markov random field.  In general, given graphs $G$ and $H$ we say that a map $\varphi:V(G)\to V(H)$ is a homomorphism if $(\varphi(u),\varphi(v))\in E(H)$ whenever $(u,v)\in E(G)$. Then  if we consider a uniform random homomorphism from $G$ to $H$, then the random variables $X_v=\varphi(v)$ form a Markov random field on $G$.

\textbf{Multivariate Gaussian distributions.} A vector of random variables $(X_1,\dots ,X_n)$ is said to have a multivariate Gaussian distribution if for any $a_1,\dots ,a_n\in \R$  the random variable $a_1X_1+\dots +a_nX_n$ has a Gaussian distribution. In this case we say that the random variables $(X_k)_{k=1}^n$ form a Gaussian random field. It turns out that a multivariate Gaussian distribution is completely described by the mean vector $\mu=(\E X_1,\dots ,\E X_n)$ and covariance matrix $\Sigma$ with entries $\Sigma_{ij}=\mathrm{Cov}(X_i,X_j)$, and so we will use the notation $\mathcal{N}(\mu,\Sigma)$ for such a Gaussian distribution. The matrix $\Sigma$ is always positive semidefinite. We say that $(X_1,\dots ,X_n)$ has a non-degenerate multivariate Gaussian distribution if $\Sigma$ is positive definite, that is, $\det(\Sigma)>0$. In this paper we only consider non-degenerate multivariate Gaussian distributions. In this case, the density function of $(X_1,\dots ,X_n)$ can be written as 
$$f(\underline{x})=\frac{1}{(2\pi)^{n/2}\det(\Sigma)^{1/2}}\exp\left(-\frac{1}{2}(\underline{x}-\mu)^T\Sigma^{-1}(\underline{x}-\mu)\right).$$
This shows that the determinant of the covariance matrix $\Sigma$ of a Gaussian random field $\mathcal{N}(0,\Sigma)$ plays a special role.  On one hand, after some transformation it acts like a partition function since
$$\int \exp\left(-\frac{1}{2}\underline{x}^T\Sigma^{-1}\underline{x}\right)\, d\underline{x}=(2\pi)^{n/2}\det(\Sigma)^{1/2}.$$
On the other hand, it also acts as an entropy. For the density function
$$f(\underline{x})=\frac{1}{(2\pi)^{n/2}\det(\Sigma)^{1/2}}\exp\left(-\frac{1}{2}\underline{x}^T\Sigma^{-1}\underline{x}\right)$$
the differential entropy is
$$-\int f(\underline{x})\ln f(\underline{x})\, d\underline{x}=\frac{n}{2}+\frac{n}{2}\ln(2\pi)+\ln \det(\Sigma).$$

Since multivariate Gaussian distributions with mean $\underline{0}$ are determined by their covariance structure, it is possible to define and study GMRF's in purely algebraic terms.
Various probabilistic properties of GMRF's (including their definitions) correspond to
simple algebraic properties of their covariance matrices. This allows us to state and prove
our main results in the more formal linear algebraic setting. While covariance matrices
in general are characterized by positive semi-definiteness, it turns out that Markov property manifests itself in two different but equivalent ways. It can be described either by a
simple determinant maximization property of the covariance matrix or by the constraint
that the inverse of the covariance matrix (also called precision matrix) has zero entries
at the non-edges of the graph. In this paper we make use of both forms. Since the determinant can be regarded both as entropy and partition function, this observation leads to two natural problems.
\bigskip

\textbf{Entropy maximization.} It is well-known that among absolutely continuous distributions on $\R$ with fixed mean $\mu$ and variance $\sigma^2$ it is the normal distribution $\mathcal{N}(\mu,\sigma^2)$ that maximizes the differential entropy. A natural variant of this question is the following: for a graph $G$ consider all Gaussian random fields such that $\mu=\underline{0}$, $\E X_v^2=1$ for all $v\in V(G)$, and $\E X_uX_v=x_e$ is a fixed value for each edge $e=(u,v)$. By the above discussion searching for the maximum entropy distribution among these Gaussian random fields  is equivalent with searching for the matrix $\widehat{\Sigma}$ with maximal determinant among positive semidefinite matrices with prescribed elements: $\Sigma_{vv}=1$ in the diagonal and $\Sigma_{u,v}=x_e$ for each edge $e=(u,v)\in E(G)$. It turns out that if there is a positive definite matrix satisfying these conditions, then the determinant maximizer $\widehat{\Sigma}$ has the property that $\widehat{\Sigma}^{-1}_{uv}=0$ whenever $(u,v)\notin E(G)$, and this is the unique one with this condition. This technical condition is equivalent with a very natural one: the Gaussian random field has the spatial Markov property! (See Theorems 2.2, 2.3 and 2.4 of \cite{RueHeld}.) We call this random field the Gaussian Markov random field. 

\textbf{Maximum likelihood estimator.}
The maximum likelihood estimator (MLE) problem for a Gaussian graphical model is the following. Given a graph $G=(V,E)$ and given samples $X_1,\dots ,X_m$ from a centered multivariate normal distribution $\mathcal{N}(\underline{0},\Sigma)$, that is, $X_k=(X^{(k)}_v)_{v\in V}$ for each $k=1,\dots ,m$. Suppose that the sample covariance matrix $S=\frac{1}{m}\sum_{k=1}^mX_kX_k^T$ can only be measured along the edges of a graph $G$.
Given this partially filled matrix $S_G$ we are looking for the maximum likelihood estimator $\widehat{\Sigma}$. Concerning this problem Dempster  \cite{Demp} proved the following theorem.

\begin{Th}[Dempster \cite{Demp}]
In the Gaussian graphical model on the graph $G$, the MLE of the
covariance matrix $\Sigma$ exists if and only if the $G$-partial sample covariance
matrix $S_G$ can be completed to a positive definite matrix. Then the MLE $\widehat{\Sigma}$
is the unique completion satisfying $\widehat{\Sigma}^{-1}_{i,j}=0$ for all $(i,j)\notin E(G)$.
\end{Th}

So again the MLE is the Gaussian Markov random field on the graph $G$.
As we already mentioned, in this paper we mainly focus on a homogeneous version of this problem, where the diagonal elements of $S_G$ are $1$, and for every edge $(i,j)$ of $G$ we have $(S_G)_{i,j}=x$ for some $x\in (-1,1)$. The formal description is as follows.
\bigskip

\begin{Def}
For a finite graph $G=(V,E)$ and $x\in (-1,1)$ let $\mathcal{A}(G,x)$ denote the set of $V\times V$ matrices $M$ such that
\begin{enumerate}
\item $M$ is positive definite,
\item every diagonal entry of $M$ is $1$,
\item $M_{i,j}=x$ for every edge $(i,j)$ of $G$,
\item $M_{i,j}$ is a priori unspecified if $i$ and $j$ are not adjacent.
\end{enumerate}
Let $\A_G(x)\in \mathcal{A}(G,x)$ be the matrix that maximizes the determinant in $\mathcal{A}(G,x)$. Let 
$$\tau(G,x):=\det(\A_G(x)).$$
\end{Def}

If $x\in (0,1)$ the set $\mathcal{A}(G,x)$ is non-empty since the matrix with $x$ for all off-diagonal elements is positive definite. Its eigenvalues are $1+(n-1)x$ with multiplicity $1$ and $1-x$ with multiplicity $n-1$. (For bipartite graphs the set $\mathcal{A}(G,x)$ is non-empty for all $x\in (-1,1)$. In general, the set $\mathcal{A}(G,x)$ is non-empty in the interval $(-\frac{1}{\vartheta(\overline{G})-1},1)$, where $\vartheta(G)$ is Lov\'asz's theta function.) So we can be sure that the determinant-maximizer exists. The strict concavity of the function $M\mapsto \log(\det(M))$ (a consequence of an entropy inequality) and the convexity of $\mathcal{A}(G,x)$ together imply that there is a unique matrix $\A_G(x)$ in $\mathcal{A}(G,x)$ which maximizes determinant. This is exactly the entropy maximizer and MLE estimator, and as we mentioned this matrix is known as the covariance matrix of the Gaussian Markov random field $\{X_v\}_{v\in V(G)}$ (or shortly GMRF) on $G$ in which $X_v\sim \mathcal{N}(0,1)$ holds for every vertex $v$ and $\mathbb{E}(X_iX_j)=x$ holds for every edge $(i,j)$ of $G$.

In this paper we focus on the function $\tau(G,x)$. This is an  analytic function of $x$ for every fixed graph $G$. One can for example easily see that if $G$ is a tree, then $\A_G(x)$ is the $V\times V$ matrix with entries $x^{d_G(u,v)}$, where $d_G(u,v)$ is the distance of the vertices $u$ and $v$ in the graph $G$. Furthermore, $\tau(G,x)=(1-x^2)^{|E(G)|}.$ 
If $G$ is the four cycle, then
$$\tau\Bigl(~C_4,\sqrt{x^2/2+x/2}~\Bigr)=1-2x+2x^3-x^4.$$
One can compute explicitly the function $\tau(G,x)$ for some very special classes of graphs like chordal graphs, strongly regular graphs and complete (bipartite) graphs, but
in general we are not aware of any nice explicit formula for $\tau(G,x)$. However, we know that the power series expansion of $\tau(G,x)$ around $0$ has integer coefficients that carry interesting combinatorial meaning. 
\bigskip

The homogeneity requirement might be strange for the first sight, but it is natural for several reasons. 
First of all, in this case the maximizing matrix $\A_G(x)$ exists. In general this is not an easy question, this is the positive definite matrix completion problem, see the papers \cite{BJL1,BJLT,LUZ,Uhler,Uhler2}. Secondly, it turns out that some of our theorems is not even true in the non-homogenous case. Thirdly, even if we only consider a Markov chain, it is most natural to consider the setting when we have a fixed transition matrix at every step, this corresponds to the case when we have the same marginal  at every edge of the path graph. Different marginals would correspond to the case that we change the transition matrix at every step.
We will also see that the function $\tau(G,x)$ can be expressed as the limit of normalized homomorphism numbers, and these are most natural if we consider the same graphon on every edge. There is also another motivation coming from graph theory. To explain it we need the concept of orthogonal representation. For a graph $G$ the vectors $v_1,\dots ,v_n\in \mathbb{R}^d$ form an orthogonal representation (OR) if the scalar product $\langle v_i,v_j \rangle=0$ for all $(i,j)\notin E(G)$, and $||v_i||_2=1$ for all $i\in V(G)$. The Lov\'asz $\vartheta$ function \cite{Lov} is then defined as follows:
$$\vartheta(G)=\min_{||c||_2=1 \atop v_i\ OR}\max_{i\in V(G)}\frac{1}{\langle v_i,c\rangle^2}.$$
One can imagine it as $c$ being the handle of an umbrella and we try to close the umbrella as much as we can without violating the conditions $\langle v_i,v_j\rangle =0$ for  $(i,j)\notin E(G)$. The role of $\tau(G,x)$ is very similar: this time we have  $v_1,\dots ,v_n\in \mathbb{R}^d$ with $||v_i||_2=1$ and $\langle v_i,v_j\rangle=x$ for $(i,j)\in E(G)$ and we try to maximize the volume of the parallelepiped determined by the vectors $v_i$. Besides the analogy there is one more connection here: $\mathcal{A}(G,x)$ is non-empty if and only if $x\in (-\frac{1}{\vartheta(\overline{G})-1},1)$, where $\overline{G}$ is the complement of the graph $G$ (see \cite{Knuth}). Nevertheless in Section~\ref{multivariate} we will explain which theorems extend to the non-homogeneous case. 
\bigskip

Using a large deviation principle on high dimensional spheres and logarithmic graph limits \cite{Sz1} we will relate the quantities $\tau(G,x)$ to more familiar subgraph densities $t(G,H)$ from extremal combinatorics. If $G$ and $H$ are finite graphs, then $t(G,H)$ denotes the probability that a random map from $V(G)$ to $V(H)$ takes edges to edges, that is, it is a homomorphism. We say that $G_1,G_2,\dots, G_n$ satisfy a multiplicative inequality with powers $\alpha_1,\alpha_2,\dots,\alpha_n$ if $$\prod_{i=1}^n t(G_i,H)^{\alpha_i}\geq 1$$ holds for every nonempty graph $H$.

For example the famous conjecture of Sidorenko \cite{Sid} says that if $G$ is a bipartite graph, then $G,K_2$ satisfies a multiplicative inequality with powers $1,-|E(G)|$. In other words, $t(G,H)\geq t(K_2,H)^{|E(G)|}$ holds for every $H$. In this case we say that $G$ is a {\bf Sidorenko graph}. Sidorenko's conjecture \cite{Sid} then says that every bipartite graph is a Sidorenko graph. Even though this conjecture is still open there are many known examples for Sidorenko graphs including rather general infinite families. For literature on Sidorenko's conjecture see: \cite{BP,Sid,Hat,L,LSz,BR,KLL,CFS,Sz1,Sz2}.
The smallest graph for which Sidorenko's conjecture is not known is the so-called M\"obius ladder on $10$ vertices which is isomorphic to $K_{5,5}\setminus C_{10}$. Note that a somewhat stronger version of the conjecture was formulated by Erd\H os and Simonovits in \cite{Sim}. Our next theorem says that the quantities $\tau(G,x)$ behave as subgraph densities in terms of multiplicative inequalities. 

\begin{Th}\label{hom-tau} If $G_1,G_2,\dots G_n$ satisfy a multiplicative inequality with powers $\alpha_1,\alpha_2,\dots,\alpha_n$, then $\prod_{i=1}^n \tau(G_i,x)^{\alpha_i}\geq 1$ holds for every $x\in (-1,1)$.
\end{Th}

To prove Theorem~\ref{hom-tau} we use a large deviation principle on high dimensional spheres (see Theorems~\ref{sphere} and \ref{ldsph}) which is interesting on its own right. 
We show that $\tau(G,x)$ is the limit (using logarithmic limits from \cite{Sz1}) of graphs arising from growing dimensional spheres. In this sense Theorem~\ref{main2} below verifies Sidorenko's conjecture in this special limiting situation. 
\bigskip

In this paper we put a great emphasis on the understanding of the function $\tau(G,x)$, especially on its estimates. The following theorem is one of the main results of this paper.

\begin{Th}\label{tau-main} Let $G$ be a graph with $e(G)$ edges, and $x\in [0,1)$. Then 
$$\tau(G,x)\geq(1-x^2)^{e(G)}.$$ 
\end{Th}

Note that since $\tau(K_2,x)=1-x^2$ for the complete graph on $2$ vertices, denoted by $K_2$, the theorem says that $\tau(G,x)\geq\tau(K_2,x)^{e(G)}$.

It is easy to see that if $G$ is bipartite, then $\tau(G,x)$ is an even function of $x$, and thus Theorem~\ref{tau-main} implies the following weaker result.

\begin{Th}\label{main2} Let $G$ be a bipartite graph with $e(G)$ edges, and $x\in (-1,1)$. Then 
$$\tau(G,x)\geq(1-x^2)^{e(G)}.$$ 
\end{Th}

As we already mentioned the importance of the function $\tau(G,x)$ is also rooted in the fact that the differential entropy of the GMRF corresponding to $\A_G(x)$ is $$\frac{|V|}{2}\ln(2\pi e)+\frac{1}{2}\ln(\tau(G,x)).$$
The next observation connects the theorems in this paper to differential entropy. Let us denote the differential entropy of a joint distribution $\{X_i\}_{i\in I}$ by $\mathbb{D}(\{X_i\}_{i\in I})$, that is,
$$\mathbb{D}(\{X_i\}_{i\in I})=-\int_{\mathbb{R}^I}f(x)\ln f(x)dx,$$
where $f(x)=f(x_i,i\in I)$ is the joint density of $\{X_i\}_{i\in I}$.  

\begin{observation}\label{obs1} If $G=(V,E)$ satisfies the inequality $\tau(G,x)\geq\tau(K_2,x)^{|E(G)|}$, then every homogeneous GMRF $\{X_v\}_{v\in V}$ on $G$ satisfies the following entropy inequality
\begin{equation}\label{entin}
\mathbb{D}(\{X_v\}_{v\in V})-\sum_{(i,j)\in E}\mathbb{D}(\{X_i,X_j\})+\sum_{v \in V} ({\rm deg}(v)-1)\mathbb{D}(X_v)\geq 0.
\end{equation}
\end{observation}

Using the homogeneity of $\{X_v\}_{v\in V}$, the inequality in the observation is equivalent with the fact that the differential entropy of the whole field is at least $|E(G)|$ times the edge entropy plus $|V(G)|-2|E(G)|$ times the vertex entropy. Formally, the vertex entropies are only needed to cancel the extra additive constant in the formula for differential entropy, however we believe that they may become important in a more general circle of questions. The left hand side of (\ref{entin}) is an interesting invariant for general GMRF's where the marginals are not necessarily equal. 
\bigskip

We will also study the logarithmic derivative of the function $\tau(G,x)$. Theorem~\ref{tau-main} asserts that the function $\frac{\tau(G,x)}{(1-x^2)^{e(G)}}$ is always at least $1$. One might wonder whether this is a monotone increasing function. After taking logarithm and differentiating this question is equivalent with the question
$$\frac{\tau'(G,x)}{\tau(G,x)}\geq e(G)\frac{\tau'(K_2,x)}{\tau(K_2,x)}.$$
It will turn out that there is a very nice identity for the logarithmic derivative in terms of the elements of the matrix $\A_G(x)^{-1}$, see Lemma~\ref{log-derivative}. This identity will have many applications in this paper. In particular, it will enable us to prove 
the following strengthening of Theorem~\ref{tau-main} for certain $x$.

\begin{Th} \label{log-ineq} Let $G$ be a graph with $e(G)$ edges, largest degree $\Delta$, average degree $\overline{d}$. Then for all $x\in [0,\frac{1}{\Delta-1}]\cup [\frac{1}{\overline{d}-1},1]$ we have
$$\frac{\tau'(G,x)}{\tau(G,x)}\geq e(G)\frac{\tau'(K_2,x)}{\tau(K_2,x)}.$$
\end{Th}

Theorem~\ref{log-ineq} is slightly deceiving as it is actually two theorems since $\tau(G,x)$ behaves differently on the two intervals as the next theorem shows. This theorem also shows the tightness or non-tightness of Theorem~\ref{tau-main}.

The girth of a graph $G$ is the length of its shortest cycle.

\begin{Th} \label{tightness}
Let $G$ be a graph with $e(G)$ edges, largest degree $\Delta$ and average degree $\overline{d}$ and girth $g$.
\medskip

\noindent (a) For $x\in [0,\frac{1}{\Delta-1})$ we have 
$$\left| \frac{\ln \tau(G,x)}{e(G)}-\ln (1-x^2)\right|\leq 2\frac{((\Delta-1)x)^g}{1-(\Delta-1)x}.$$
(b) For $x>\frac{1}{\overline{d}-1}$ there exists a positive function $\alpha(\overline{d},x)$ independently of $G$ such that
$$\frac{\ln \tau(G,x)}{e(G)}-\ln (1-x^2)\geq \alpha(\overline{d},x).$$
\end{Th} 

Note that part (a) immediately implies the following theorem. 

\begin{Th} \label{large girth limit}
Let $(G_n)_n$ be a sequence of $d$--regular graphs with girth $g(G_n)\to \infty$. Let $v(G)$ denote the number of vertices of $G$. Assume that $x\in [0,\frac{1}{d-1})$. Then
$$\lim_{n\to \infty}\frac{\ln \tau(G_n,x)}{v(G_n)}=\frac{d}{2}\ln(1-x^2).$$
\end{Th}

In words, part (a) of Theorem~\ref{tightness} and Theorem~\ref{large girth limit} show that for small $x$, the bound provided by Theorem~\ref{tau-main} is asymptotically tight for large girth graphs. On the other hand, for large $x$ the bound cannot be tight even for large girth graphs. 
We will refer to the interval $[0,\frac{1}{\Delta-1}]$ as the first interval, and to the interval $[\frac{1}{\overline{d}-1},1]$ as the second interval. Note that if $G$ is regular, then $[0,\frac{1}{\Delta-1}]\cup [\frac{1}{\overline{d}-1},1]=[0,1]$, and so Theorem~\ref{log-ineq} is a strengthening of Theorem~\ref{tau-main}. For regular graphs Theorem~\ref{tightness} shows a phase transition at the point $\frac{1}{d-1}$. This point is also related to the number of spanning trees. For instance, we will give a new proof for  McKay's  upper bound (\cite{mckay}) to the number of spanning trees of regular graphs, see Theorem~\ref{McKay's bound} below.
\bigskip

We offer two more theorems about regular graphs. A simple application of Theorem~\ref{hom-tau} is the following result.

\begin{Th} \label{Kdd-tau}
Let $v(G)$ denote the number of vertices of $G$. Then for any $d$--regular bipartite graph $G$ we have
$$\tau(G,x)^{1/v(G)}\leq \tau(K_{d,d},x)^{1/(2d)}$$
for all $x\in (-1,1)$.
\end{Th}

As a counterpart of Theorem~\ref{Kdd-tau} we prove that among all regular graphs the complete graph maximizes the quantity $\tau(G,x)^{1/v(G)}$.

\begin{Th} \label{regular-complete} Let $v(G)$ denote the number of vertices of $G$. Then for any $d$--regular graph $G$ and $x\in (0,1)$ we have
$$\tau(G,x)^{1/v(G)}\leq \tau(K_{d+1},x)^{1/v(K_{d+1})}.$$
\end{Th}

\newpage

\textbf{Notations.} We use standard notations. If $G=(V,E)$ is a graph, then $|V(G)|=v(G)$ denotes the number of vertices,  and $|E(G)|=e(G)$ denotes the number of edges. In general, when $G$ is fixed, then $n$ will also denote the number of vertices. The largest degree will be denoted by $\Delta$, while $\overline{d}$ denotes the average degree. The set of neighbors of a vertex $u$ is denoted by $N_G(u)$. $K_n$ denotes the complete graph on $n$ vertices. Similarly, $K_{a,b}$ denotes the complete bipartite graph with parts of size $a$ and $b$.

For a vector $x=(x_1,\dots ,x_n)\in \R^n$ the norm $\|x\|_p=\left(\sum_{i=1}^n|x_i|^p\right)^{1/p}$.
$\langle u,v\rangle$ denotes the scalar product of the vectors $u,v\in \R^n$, and $\|u\|_2=\langle u,u\rangle^{1/2}$ is the norm of a vector $u$. We also use the notation $\mathbb{S}_{n-1}=\{x\ |\ x\in \mathbb{R}^n,\  \|x\|_2=1\}$. The Hadamard product of the $n\times n$ matrices $A$ and $B$ is denoted by $A\circ B$, where $(A\circ B)_{ij}=A_{ij}B_{ij}$.

\bigskip

\textbf{This paper is organized as follows.} First of all,  we call attention to the fact that in the appendix one can find all frequently used tools from probability theory and matrix analysis. 

In Section~\ref{basic-properties} we introduce some basic properties of $\tau(G,x)$, the matrix $\A_G(x)$, and we set up an optimization problem for the inverse  matrix of $\A_G(x)$, this will be a key tool throughout the paper. 

In Section~\ref{large-deviation} we prove a large deviation result for random vectors chosen from a large dimensional sphere.

In Section~\ref{homomorphism} we use  results of Section~\ref{large-deviation} to connect our theory with graph homomorphisms. In particular, we prove Theorems~\ref{hom-tau} and \ref{Kdd-tau} in this section. 

In Sections~\ref{conditional}  we give a useful lemma that provides an algorithm to compute numerically $\tau(G,x)$ and $\A_G(x)$ with high precision. 

In Section~\ref{operations} we prove Theorem~\ref{tau-main} through the study of certain graph operations. Theorem~\ref{main2} immediately follows from this theorem.

In Section~\ref{second-interval} we prove Theorems~\ref{log-ineq} and \ref{tightness} for the second interval. 

In Section~\ref{M-graphs} we prove Theorem~\ref{log-ineq} for the first interval and Theorem~\ref{regular-complete}.  

In Section~\ref{Ihara-zeta_section} we relate $\tau(G,x)$ to the so-called Ihara zeta function and  prove Theorem~\ref{tightness} for the first interval.

In Section~\ref{spanning tree} we use our theory to give a new proof of a result of B. McKay on the number of spanning trees of regular graphs. 

In Section~\ref{multivariate} we give the versions of our theorems in the multivariate case without proof. 

In Section~\ref{open problems} we give some open problems.

Section~\ref{preliminaries} contains useful results about probability theory and matrix analysis that we use throughout the paper.
\bigskip

\begin{tikzpicture}[node distance=1.5cm,
    every node/.style={fill=white, font=\sffamily}, align=center]
\node (Sect2) [base] {Section 2 \\  Basics};
\node (Sect5) [base, below of=Sect2] {Section 6 \\ Sidorenko-type inequality};
\node[xshift=-3.2cm] (Sect3) [base, left of=Sect5] {Section 3 \\ Large deviation principle};
\node[xshift=3.2cm] (Sect6) [base, right of=Sect5] {Section 7 \\ Logarithmic derivative};
\node[xshift=-0.6cm] (Sect4) [base, below of=Sect3] {Section 4 \\ Homomorphisms};
\node (Sect4.5) [base, below of=Sect4] {Section 5 \\ Recoupling algorithm};
\node[xshift =-1.7 cm] (Sect9) [base, below of=Sect5] {Section 9 \\ Ihara zeta function};
\node[xshift =1.7 cm] (Sect10) [base, below of=Sect5] {Section 10 \\ Spanning trees};
\node[xshift =0.4 cm] (Sect8) [base, below of=Sect6] {Section 8 \\ Total positivity};
\node[xshift =2.2 cm] (Sect11) [base, right of=Sect4.5] {Section 11 \\ Multivariate case};
\node[xshift =1.8 cm] (Sect12) [base, right of=Sect11] {Section 12 \\ Open problems};
\node (Sect13) [base, below of=Sect8] {Section 13 \\ Appendix};
\draw (Sect2) -- (Sect3);
\draw (Sect2) -- (Sect5);
\draw (Sect2) -- (Sect6);
\draw (Sect3) -- (Sect4);
\draw (Sect5) -- (Sect9);
\draw (Sect5) -- (Sect10);
\draw (Sect6) -- (Sect8);

\end{tikzpicture}

\newpage

\section{Basic properties of $\A_G(x)$ and $\tau(G,x)$} \label{basic-properties}

\subsection{The matrix $\A_G(x)$.} First we study the matrix $\A_G(x)$. The following is just Dempster's theorem. For sake of convenience we include a proof.

\begin{Th} \label{unique-maximizer} The determinant maximization problem over $\mathcal{A}(G,x)$ has a unique maximizer $\A_G(x)$. If $u$ and $v$ are not adjacent vertices, then $\A_G(x)^{-1}_{u,v}=0.$
\end{Th}

\begin{proof} The set $\mathcal{A}(G,x)$ is convex, and it is non-empty since the matrix whose all diagonal elements are $1$ and off-diagonal elements are $x$ is positive definite. It is also bounded since for any matrix $A\in \mathcal{A}(G,x)$ we have $|A_{u,v}|\leq 1$ to make sure that $\det \left(\begin{array}{cc} A_{uu} & A_{uv} \\ A_{vu} & A_{vv}\end{array}\right)=\det \left(\begin{array}{cc} 1 & A_{uv} \\ A_{vu} & 1\end{array}\right)\geq 0$. The closure of $\mathcal{A}(G,x)$ is thus compact. So the function $\det$ has a maximum on it. The maximum cannot be achieved on the boundary since $\det\equiv 0$ on it. The function $\det$ is also strictly log-concave on $\mathcal{A}(G,x)$ (see Theorem~\ref{log-concave-pos-def}), so the maximizer is unique.

For non-adjacent vertices $u$ and $v$ consider the matrix $F_{u,v}$ which takes value $1$ at the entries $(u,v)$ and $(v,u)$ and $0$ everywhere else.  Then for small enough $t$, the matrix $A(t)=\A_G(x)+tF_{u,v}\in \mathcal{A}(G,x)$. We also have
$$\det A(t)=d_0+d_1t+d_2t^2$$
for some $d_0,d_1,d_2$. Since $\det A(t)\leq \det \A_G(x)=d_0$ for all $t$ in an open neighborhood of $0$ we get that $d_1=0$ and $d_2\leq 0$.
Let $\B_G(x)$ be the inverse of $\A_G(x)$. Then
$$d_1=2\det(\A_G(x))\B_G(x)_{u,v}.$$
Hence $\B_G(x)_{u,v}=0$ for all non-adjacent vertices $u$ and $v$. 

\end{proof}

\begin{Def}
We will denote the inverse of $\A_G(x)$ by $\B_G(x)$ throughout the whole paper. It will be convenient to parametrize $\B_G(x)$ as follows: $\B_G(x)_{u,v}=-y_G(u,v)$ if $(u,v)\in E(G)$, $\B_G(x)_{u,u}=1+x\sum_{v\in N(u)}y(u,v)$. (We have seen that $\B_G(u,v)=0$ if $(u,v)\notin E(G)$. We further study $\B_G(x)$ in Section~\ref{dual-optimization}.) We will also use the notation $z_G(u,v)=\A_G(x)_{u,v}$. Thus $z_G(u,u)=1, z_G(u,v)=x$ if $(u,v)\in E(G)$. 
\end{Def} 

\subsection{Properties of the function $\tau(G,x)$.}

\begin{Lemma}
The function $\tau(G,x)$ is monotone decreasing logarithmically concave function on the interval $(0,1)$.
\end{Lemma}

\begin{proof} First we prove the logarithmic concavity. Observe that for $\alpha\in (0,1)$ the matrix $\alpha\A_G(x)+(1-\alpha)\A_G(y)\in \mathcal{A}(G,\alpha x+(1-\alpha) y)$. Hence
\begin{align*}
\tau(G,\alpha x+(1-\alpha) y)&\geq \det(\alpha\A_G(x)+(1-\alpha)\A_G(y))\\
                             &\geq (\det \A_G(x))^{\alpha}(\det \A_G(y))^{1-\alpha}\\
														 &=\tau(G,x)^{\alpha}\tau(G,y)^{1-\alpha},
\end{align*}
where we used the fact that $\det$ is a logarithmically concave function on the set of positive definite matrices (Theorem~\ref{log-concave-pos-def}).

To prove the monotonicity we observe that $\A(x)\circ \A(y)\in \mathcal{A}(G,xy)$, where $\circ$ is the Hadamard product of two matrices. Combining this fact with Oppenheim's inequality (Theorem~\ref{Opponheim}) we get that
$$\tau(G,xy)\geq \det(\A_G(x)\circ \A_G(y))\geq \det(\A_G(x))\prod_{i=1}^n\A_G(y)_{ii}=\det(\A_G(x))=\tau(G,x).$$
This proves the monotonicity. 

\end{proof}

\begin{Rem}
Another notable inequality of Oppenheim asserts that for positive semidefinite matrices $A$ and $B$ we have
$$\det(A)\prod_{i=1}^nb_{ii}+\det(B)\prod_{i=1}^na_{ii}\leq \det(A\circ B)+\det(A)\det(B).$$
Using this inequality to $\A_G(x)$ and $\A_G(y)$ we get that
$$\tau(G,x)+\tau(G,y)\leq \tau(G,xy)+\tau(G,x)\tau(G,y).$$
\end{Rem}

\begin{Lemma} \label{bipartite-symmetry}
For a bipartite graph $G=(A,B,E)$ we have $\tau(G,x)=\tau(G,-x)$ for $x\in (-1,1)$.
\end{Lemma}

\begin{proof}
Let $D$ be the $V\times V$ diagonal matrix having value $1$ at the elements corresponding to the vertices of $A$, and $-1$ at the elements corresponding to the vertices of $B$. Then  $\Psi: M \mapsto DMD$ maps $\mathcal{A}(G,x)$ to $\mathcal{A}(G,-x)$ bijectively such that $\det(M)=\det(\Psi(M))$. This immediately implies the claim of the lemma.
\end{proof}

\subsection{Dual optimization problem} \label{dual-optimization}

In this part we set up a dual optimization problem for $\B_G(x)$. Indeed, the following dual optimization problem holds. 

\begin{Lemma} \label{dual} Let $\mathcal{B}(G,x)$ be the set of positive definite matrices $B=B(\underline{t})$ which are parametrized as follows:
\begin{enumerate}
\item if $(u,v)\notin E(G)$, then $B_{u,v}=0$,
\item if $(u,v)\in E(G)$, then $B_{u,v}=-t(u,v)$,
\item for $u\in V(G)$ we have $B_{u,u}=1+x\sum_{v\in N_G(u)}t(u,v)$.
\end{enumerate}
Then $\det(B)$ is a strictly log-concave function on $\mathcal{B}(G,x)$, and it takes its maximum at the unique $B(\underline{y})$ for which $B(\underline{y})^{-1}=\A_G(x)$, i. e., $B(\underline{y})=\B_G(x)$.
\end{Lemma}

Lemma~\ref{dual} follows from general duality theorems of positive definite programming, but in this special case we can give a self-contained simple proof.

\begin{proof} For a $B\in \mathcal{B}(G,x)$ and a vertex $u\in V(G)$ we have 
$$(B\cdot \A_G(x))_{uu}=1.$$
Consequently, $\mathrm{Tr}(B\cdot \A_G(x))=n$.
The matrix $B\cdot \A_G(x)$ is not necessarily symmetric, so it may not be positive definite. So let us consider the matrix $C=B^{1/2}\A_G(x)B^{1/2}$ which is symmetric and positive definite.
Then 
$$\det(B\cdot \A_G(x))=\det(C)\leq \left(\frac{\mathrm{Tr}(C)}{n}\right)^n=\left(\frac{\mathrm{Tr}(B\cdot \A_G(x))}{n}\right)^n=1.$$
Hence
$$\det(B)\leq (\det \A_G(x))^{-1}=\det \B_G(x).$$
We have seen that $\B_G(x)\in \mathcal{B}(G,x)$ (see Theorem~\ref{unique-maximizer}). The function $\det$ is strictly log-concave on the set of positive definite matrices. The set $\mathcal{B}(G,x)$ is convex. So the function $\det$  has a unique maximizer on $\mathcal{B}(G,x)$ which must be $\B_G(x)$.

\end{proof}

\subsection{Equations for $\A_G(x)$ and $\B_G(x)$}

Recall that we use the notation $z_G(u,v)=\A_G(x)_{u,v}$, and $y_G(u,v)=-\B_G(x)_{u,v}$ if $(u,v)\in E(G)$. Sometimes we drop $G$ from the subscript if it is clear from the context.

By the parametrization we get that  $(\A_G(x)\cdot \B_G(x))_{u,u}=1$ is automatically satisfied. From $(\A_G(x)\cdot \B_G(x))_{u,w}=0$ for $u\neq w$ we get that for $(u,w)\in E(G)$ we have
$$x=(1-x^2)y_G(u,w)+\sum_{{v\in N_G(w)} \atop {v\neq u}}(z_G(u,v)-x^2)y_G(v,w),$$
and for $(u,w)\notin E(G)$ we get that
$$z_G(u,w)=\sum_{v\in N_G(w)}(z_G(u,v)-xz_G(u,w))y_G(v,w).$$
In the latter case we can rewrite it as
$$z_G(w,u)=\sum_{v\in N_G(w)}z_G(v,u)\frac{y_G(v,w)}{1+x\sum_{r\in N_G(w)}y_G(r,w)}.$$


\begin{Rem}
These equations enable us to compute the functions $z_G(u,v)$ and $y_G(u,v)$ for several nice families of graphs like complete bipartite graphs and strongly regular graphs.
\end{Rem}

\section{Large deviation principle} \label{large-deviation}

It is well known that if $n\in\mathbb{N}$ is a fixed number and $k$ is big compared to $n$, then if we choose independent uniform vectors $v_1,v_2,\dots,v_n$ in the sphere $\mathbb{S}_{k-1}=\{x\ |\ x\in\mathbb{R}^k, \|x\|_2=1\}$, then with probability close to one the vectors are close to be pairwise orthogonal. It will be important for us to estimate the probability of the atypical event that the scalar product matrix $\langle v_i,v_j\rangle_{1\leq i,j\leq n}$ is close to some matrix $A$ that is separated from the identity matrix. Let $\lambda_n$ denote the Lebesgue measure on the space of symmetric $n\times n$ matrices with $1$'s in the diagonal. Since this space can be identified with $\mathbb{R}^{n(n-1)/2}$, the Lebesgue measure $\lambda_n$ is just the usual measure on that space. In this section we give a simple formula for the density function $\langle v_i,v_j\rangle_{1\leq i,j\leq n}$ relative to the Lebesgue measure $\lambda_n$. Using this formula we prove a large deviation principle for the scalar product matrices of random vectors. 

The following theorems and proofs extensively use some results on chi $(\chi)$, chi-square $(\chi^2)$, Wishart distributions, multivariate gamma function. All these concepts and results can be found in the Appendix. 

\begin{Th}\label{sphere} Assume that $k\geq n\geq 2$ are integers. Let $v_1,v_2,\dots,v_n$ be independent, uniform random elements of the sphere $\mathbb{S}_{k-1}$ and let $M(n,k)$ be the $n\times n$ matrix with entries $M(n,k)_{i,j}:=\langle v_i,v_j\rangle$. The probability density function $f_{n,k}$ of $M(n,k)$ is supported on the set $\mathcal{M}_n$ of positive semidefinite $n\times n$ matrices with $1$'s in the diagonal entries and is given by the formula $$f_{n,k}(M)=\det(M)^{(k-n-1)/2}\Gamma(k/2)^n\Gamma_n(k/2)^{-1}$$ where $\Gamma_n$ is the multivariate $\Gamma$-function.
\end{Th}

\begin{proof} For $1\leq i\leq n,1\leq j\leq k$ let $Z_{ij}$ be independent standard normal distributions. Let $X_i=(\sum_{j=1}^kZ_{ij}^2)^{1/2}$. Then $\{X_i\}_{i=1}^n$ is a system of $n$ independent random variables with $\chi_k$ distributions. Set $v_i=\frac{1}{X_i}(Z_{i1},\dots ,Z_{ik})$. Then $v_1,\dots ,v_n$ are uniform random elements of the sphere $\mathbb{S}_{k-1}$, because of the spherical symmetry of the $k$ dimensional standard normal distribution. If $G$ denotes the matrix with elements $Z_{ij}$, then $M'(n,k)=GG^T$ is from the Wishart distribution corresponding to the $n\times n$ identity matrix. It is also the matrix with elements $\langle X_iv_i,X_jv_j\rangle$. It follows that the density function $\tilde{f}_{n,k}$ of $M'(n,k)$ is supported on positive semidefinite matrices and is given by 
$$\tilde{f}_{n,k}(M)=\det(M)^{(k-n-1)/2}e^{-{\rm Tr}(M)/2}2^{-kn/2}\Gamma_n(k/2)^{-1}.$$
The next step is to compute the conditional distribution of $M'(n,k)$ in the set $M'(n,k)_{i,i}=X_i^2=1$ for $1\leq i\leq n$. Since  $X_1,\dots ,X_n$ is a system of $n$ independent random variables with $\chi_k$ distributions and by the fact that the density function $g_k(x)$ of $\chi_k^2$ is
$$g_k(x)=x^{k/2-1}e^{-x/2}2^{-k/2}\Gamma(k/2)^{-1}$$
the statement of the proposition follows from $f_{n,k}(M)=\tilde{f}_{n,k}(M)/g_k(1)^n$ for $M\in\mathcal{M}_n$ and $f_{n,k}(M)=0$ for $M\notin\mathcal{M}_k$. 
\end{proof}

\begin{Rem} It is a nice fact that Theorem~\ref{sphere} allows us to give an explicit formula for the volume of the spectahedron $\mathcal{M}_n$. (In general,  the set of $n \times n$ positive semidefinite matrices forms a convex cone in $\R^{n\times n}$, and a spectrahedron is a shape that can be formed by intersecting this cone with a linear affine subspace.) If $k=n+1$, then $f_{n,k}$ is a constant function and by the fact that it is a density function, this constant is the inverse volume of $\mathcal{M}_n$. We obtain that $\lambda_n(\mathcal{M}_n)=\Gamma((n+1)/2)^{-n}\Gamma_n((n+1)/2)$.
\end{Rem}

\begin{Lemma}\label{asym} For $n\geq 2$ we have that $$\lim_{n\to\infty}\frac{\Gamma(k/2)^n\Gamma_n(k/2)^{-1}}{(k/(2\pi))^{n(n-1)/4}}=1.$$
\end{Lemma}

\begin{proof}
Set $c_r=\pi^{-1/2}\Gamma(r/2)/\Gamma((r-1)/2)$. It is straightforward from the formulas that
$$\Gamma(k/2)^n\Gamma_n(k/2)^{-1}=c_k^{n-1}c_{k-1}^{n-2}\dots c_{k-n+2}.$$ 
 It is well known that $\lim_{r\to\infty}\Gamma(r)\Gamma(r-\alpha)^{-1}r^{-\alpha}=1$. It follows that $\lim_{r\to\infty} c_r(r/(2\pi))^{-1/2}=1$ which completes the proof.
\end{proof}

Now we are ready to formulate and prove our large deviation principle. Let us denote by $\mu_{n,k}$ the probability measure corresponding to the random matrix model $M(n,k)$ defined in Theorem~\ref{sphere}. We have that $\mu_{n,k}$ is concentrated on the closed set $\mathcal{M}_n$. If $k\geq n\geq 2$, then $\mathcal{M}_n$ is a compact convex set of positive measure in the space of symmetric $n\times n$ matrices with ones in the diagonal. For a measurable function $f:\mathcal{M}_n\rightarrow\mathbb{R}$ we denote by $\|f\|_\infty$ the essential supremum of $f$ relative to the measure $\lambda_n$, that is, 
$$\|f\|_{\infty}=\inf\{t>0\ |\ \lambda_n\left(M\in \mathcal{M}_n: |f(M)|>t\right)=0\}.$$
Note that $\|f\|_\infty$ can differ from $\sup_{M\in\mathcal{M}_n}f(M)$ because changes in $f$ on $0$ measure sets are ignored. In general will use the norms $\|.\|_p$ for functions on $\mathcal{M}_n$.

\begin{Th}[Large deviation principle on the sphere]\label{ldsph} Let $n\geq 2$ be a fixed integer. Let $A\subseteq\mathcal{M}_n$ be a Borel measurable set. We have that
$$ \lim_{k\to\infty} \frac{1}{k}\ln(\mu_{n,k}(A))=\frac{1}{2}\ln\|1_A\det\|_\infty.$$
\end{Th}

\begin{proof} Set $c_{n,k}=\Gamma(k/2)^n\Gamma_n(k/2)^{-1}$. From Theorem~\ref{sphere} we have
$$\mu_{n,k}(A)=c_{n,k}\int_{A}(\det)^{(k-n-1)/2} d\lambda_n=c_{n,k}\int_{\mathcal{M}_n}1_A(\det)^{(k-n-1)/2} d\lambda_n=c_{n,k}\|1_A\det\|_{(k-n-1)/2}^{(k-n-1)/2}.$$
It follows that
$$\frac{1}{k}\ln(\mu_{n,k}(A))=\ln(c_{n,k}^{1/k})+\frac{k-n-1}{2k}\ln\|1_A\det\|_{(k-n-1)/2}.$$
From Lemma \ref{asym} we get that
$$\lim_{k\to\infty} \ln(c_{n,k}^{1/k})=0.$$
Now the statement of the theorem follows from $\lim_{p\to\infty}\|1_A\det\|_p=\|1_A\det\|_\infty$.
\end{proof}

\section{Homomorphism and spherical graphons} \label{homomorphism}

A {\bf graphon} (see \cite{LS}) is a symmetric measurable function of the form $W:\Omega^2\rightarrow [0,1]$ where $(\Omega,\mu)$ is a standard probability space. If $G$ is a finite graph, then it makes sense to introduce the "density" of $G$ in $W$ using the formula 
$$t(G,W)=\int_{x\in\Omega^{V(G)}}\prod_{(i,j)\in E(G)}W(x_i,x_j)~ d\mu^n.$$ 
Note that the conjecture of Sideronko was originally stated in this integral setting,  and it says that $t(G,W)\geq t(K_2,W)^{e(G)}$ holds for every bipartite graph $G$ and graphon $W$.

In this section we prove Theorem~\ref{hom-tau} using special graphons that we call {\bf spherical graphons}.
Let $S\subseteq [-1,1]$ be a Borel measurable set and let $n$ be a natural number. Let us define the graphon ${\rm Sph}_{S,k}:\mathbb{S}_k\times \mathbb{S}_k\rightarrow [0,1]$ such that
${\rm Sph}_{S,k}(x,y)=1$ if $\langle x,y\rangle \in S$ and ${\rm Sph}_{S,k}(x,y)=0$ if $\langle x,y\rangle\notin S$.

For a Borel measurable set $S\subseteq [-1,1]$ and a graph $G$ let $\mathcal{A}(G,S)$ denote the set of positive semidefinite $V(G)\times V(G)$ matrices $M$ such that the diagonal entries of $M$ are all $1$'s and $M_{i,j}\in S$ holds for every $(i,j)\in E(G)$. It is clear that using the notation from the previous section we have that
$$t(G,{\rm Sph}_{S,k-1})=\mu_{v(G),k}(\mathcal{A}(G,S)) $$
It follows from Theorem~\ref{ldsph} that
\begin{equation}\label{limitdens}
\lim_{k\to\infty} \frac{1}{k}\ln(t(G,{\rm Sph}_{S,k})=\frac{1}{2}\ln\|1_{\mathcal{A}(G,S)}\det\|_{\infty}.
\end{equation}
Now we get the next lemma.

\begin{Lemma}\label{denslem} Assume that $S\subseteq [-1,1]$ is a Borel set and $G$ is a Sidorenko graph.
Then
$$\|1_{\mathcal{A}(G,S)}\det\|_{\infty}\geq \|1_{\mathcal{A}(K_2,S)}\det \|_{\infty}^{e(G)}.$$
\end{Lemma}

\begin{proof} 
The Sidorenko property of $G$ implies that for every $k$ we have that
$$\frac{1}{k}\ln(t(G,{\rm Sph}_{S,k})\geq e(G)\cdot \frac{1}{k}\ln(t(K_2,{\rm Sph}_{S,k})).$$
Then (\ref{limitdens}) completes the proof by taking the limit $k\to\infty$.
\end{proof}

\begin{Lemma} \label{hom-tau-lemma} Let $S_{\varepsilon}:=[x-\varepsilon,x+\varepsilon]\cap (-1,1)$. Then
$$\frac{1}{2}\ln \tau(G,x)=\lim_{\varepsilon\to 0}\lim_{k\to \infty}\frac{1}{k}\ln(t(G,{\rm Sph}_{S_{\varepsilon},k})).$$
\end{Lemma}

\begin{proof} We know that
$\lim_{\epsilon\to 0}\|1_{\mathcal{A}(G,S_\epsilon)}\det\|_\infty=\tau(G,x)$ holds for every graph $G$.
\end{proof}

\begin{proof}[Proof of Theorems~\ref{hom-tau} and \ref{Kdd-tau}.]
Theorem~\ref{hom-tau} is an immediate consequence of Lemma~\ref{hom-tau-lemma}.
Theorem~\ref{Kdd-tau} follows from Theorem~\ref{hom-tau} and the following theorem of Galvin and Tetali \cite{GalTet}:
for any $d$--regular bipartite graph $G$ and any graph $H$ we have
$$\hom(G,H)^{1/v(G)}\leq \hom(K_{d,d},H)^{1/(2d)}.$$
\end{proof}

\section{Conditional independent couplings} \label{conditional}

In this section we review conditional independent couplings of Gaussian fields. 
On the one hand, this will lead to an efficient numeric algorithm to compute $\A_G(x)$. And on the other hand, it will enable us to understand $\tau(G,x)$ for graphs $G$ that are obtained from another graphs by an operation called clique sum, for details see Theorem~\ref{clique-sum}. The following lemma is the main tool in this section.

\begin{Lemma}\label{linalg}
Assume that $X$ and $Y$ are two finite sets with $X\cap Y=Z$ and $X\cup Y=Q$. Assume furthermore $A\in\mathbb{R}^{X\times X}$ and $B\in\mathbb{R}^{Y\times Y}$ are two positive definite matrices such that their submatrices  $A_{Z\times Z}$ and $B_{Z\times Z}$ are equal to some matrix $C\in\mathbb{R}^{Z\times Z}$. Let $\tilde{A},\tilde{B}$ and $\tilde{C}$ be the matrices in $\mathbb{R}^{Q\times Q}$ obtained from $A^{-1},B^{-1}$ and $C^{-1}$ by putting zeros to the remaining entries. Then the matrix
$$D:=(\tilde{A}+\tilde{B}-\tilde{C})^{-1}$$ satisfies the following conditions.
\begin{enumerate}
\item $D_{X\times X}=A$ , $D_{Y\times Y}=B$,
\item $D$ is positive definite,
\item $\det(D)=\det(A)\det(B)\det(C)^{-1}$.
\end{enumerate}
\end{Lemma}

\begin{proof} The statement can be checked with elementary linear algebraic methods. To highlight the connection to probability theory we give a probabilistic proof which is also more elegant. We can regard $A,B$ and $C$ as covariance matrices of Gaussian distributions $\mu_A,\mu_B$ and $\mu_C$ on $\mathbb{R}^X,\mathbb{R}^Y$ and $\mathbb{R}^Z$ with density functions $f_A,f_B$ and $f_C$. The condition $A_{Z\times Z}=B_{Z\times Z}$ is equivalent with the fact that the marginal distribution of both $\mu_A$ and $\mu_B$ on $\mathbb{R}^Z$ is equal to $\mu_C$. The conditional independent coupling of $\mu_A$ and $\mu_B$ over the marginal $\mu_C$ has density function
$$f(v)=(2\pi)^{-|Q|/2}e^{-\frac{1}{2}(vP_XA^{-1}P_Xv^T+vP_YB^{-1}P_Yv^T-vP_ZC^{-1}P_Zv^T)}$$
where $P_X,P_Y$ and $P_Z$ are the projections to the coordinates in $X,Y$ and $Z$.
We have by the definition of $D$ that
$$f(v)=(2\pi)^{-|Q|/2}e^{-\frac{1}{2} vD^{-1}v^T}$$
and thus $f$ is the density function of the Gaussian distribution $\mu$ with covariance matrix $D$.
The first property of $D$ follows from the fact that the marginals of $\mu$ on $X$ and $Y$  have covariance matrices $A$ and $B$.
The second property of $D$ follows from the fact that it is a covariance matrix of a non-degenerate Gaussian distribution. The third property follows from the fact that
$$0=\mathbb{D}(\mu)-\mathbb{D}(\mu_A)-\mathbb{D}(\mu_B)+\mathbb{D}(\mu_C)$$ holds for the differential entropies in a conditionally independent coupling. On the other hand the right hand side is equal to
$$\ln(\det(D))-\ln(\det(A))-\ln(\det(B))+\ln(\det(C)).$$
\end{proof}

We will refer to the matrix $D$ as the \emph{conditionally independent coupling} of $A$ and $B$ (over $C$) and we denote it by $A\curlyvee B$.

\begin{Th} \label{clique-sum}
Let $G_1$ and $G_2$ be two graphs. Assume that $S_1\subseteq V(G_1)$ and $S_2\subseteq V(G_2)$ are subsets such that the induced subgraphs $G_1[S_1]$ and $G_2[S_2]$  are cliques of size $k$. Let $\varphi:S_1 \to S_2$ be a bijection, and $G=G_1+_{\varphi}G_2$ be the graph obtained from $G_1\cup G_2$ by identifying vertex $v\in S_1$ with $\varphi(v)\in S_2$. Then $\A_G(x)=\A_{G_1}(x)\curlyvee \A_{G_2}(x)$. In particular
$$\tau(G,x)=\frac{\tau(G_1,x)\tau(G_2,x)}{\tau(K_k,x)}$$
for all $x$ for which $\tau(G_1,x)$ and $\tau(G_2,x)$ make sense. Furthermore, for an $e=(u,v)\in E(G)$ we have
$$y_G(u,v)=y_{G_1}(u,v)+y_{G_2}(u,v)-y_{K_k}(u,v)$$
if $u$ and $v$ are both in the common clique $K_k$ of $G_1$ and $G_2$, and is equal to $y_{G_1}(u,v)$ or $y_{G_2}(u,v)$ if $(u,v)\in E(G_1)\setminus E(G_2)$ or $(u,v)\in E(G_2)\setminus E(G_1)$, respectively.
\end{Th}

\begin{proof}
Let us apply Lemma~\ref{linalg} with $X=V(G_1)$, $Y=V(G_2)$, $Z=V(G_1)\cap V(G_2)$ and $A=\A_{G_1}(x),B=\A_{G_2}(x)$ with their common intersection $C=\A_{K_k}(x)$. The latter matrix is fixed by the constraints that $G_1[S_1]$ and $G_2[S_2]$ determine cliques. Consider the matrix
$D=\A_{G_1}(x) \curlyvee \A_{G_2}(x)$. By construction $D^{-1}=\widetilde{\B_{G_1}(x)}+\widetilde{\B_{G_2}(x)}-\widetilde{\B_{K_k}(x)}\in \mathcal{B}(G,x)$. By the first claim of Lemma~\ref{linalg} we also know that $D\in \mathcal{A}(G,x)$. By the primal and dual optimization programs we know that there is only one matrix $F$ such that $F\in \mathcal{A}(G,x)$ and $F^{-1}\in \mathcal{B}(G,x)$, that is, $F=\A_G(x)$. Hence $D=\A_G(x)$.
Having $\A_G(x)=\A_{G_1}(x)\curlyvee \A_{G_2}(x)$, the rest of the claims follow from Lemma~\ref{linalg} and the construction.
\end{proof}

\begin{Rem}
We can use the above theorem to compute $\tau(G,x)$ for chordal graphs as they can be built  up using clique sums. Another application of the above theorem is to show a graph $G$ that contains an edge $(u,v)$ such that $y_G(u,v)<0$. Let us glue together $k$ triangles at a common edge, i. e., this is the complete multipartite graph $K_{1,1,k}$. It is often called a triangular book graph. Then for the common edge $(u,v)$ we have
\begin{align*}
y_G(u,v)&=ky_{K_3}(u,v)-(k-1)y_{K_2}(u,v)\\
&=k\frac{x}{(1-x)(1+2x)}-(k-1)\frac{x}{1-x^2}\\
&=\frac{x}{1-x}\cdot \frac{1-(k-2)x}{(1+x)(1+2x)}.
\end{align*}
This is negative if $x>\frac{1}{k-2}$.

\end{Rem}

\subsection{The recoupling algorithm}\label{alg}

We finish this section with an algorithm that we used in computer experiments to approximate the matrix $\A_G(x)$ for various graph $G$ and $x\in (0,1)$. The algorithm can also be used to compute the coefficients in the power series expansion of $\tau(G,x)$.

Let $G=(V,E)$ be a fixed graph and $x\in [0,1)$. Let $M_0(G,x)$ denote the $V\times V$ matrix with $1's$ in the diagonal and $x$ elsewhere. Our algorithm produces a sequence of matrices $M_i(G,x)$ recursively with increasing determinants such that they converge to $\A_G(x)$. 

Then we recursively produce a matrix $M_{i+1}(G,x)$ from $M_i(G,x)$ by choosing $(v,w)\notin E(G)$. Let $A:= V\setminus{v}$, $B:=V\setminus{w}$ and $C:=V\setminus \{v,w\}$. Then we set $$M_{i+1}(G,x):=M_i(G,x)_{A\times A}\curlyvee M_i(G,x)_{B\times B}.$$

The algorithm depends on a choice of non-edges $(v_1,w_1),(v_2,w_2),\dots$. Our choice is to repeat a fix ordering of all non-edges several times. One can also perform the algorithm with formal matrices in which the entries are rational functions of $x$. It is easy to see by induction that in each step the entries remain of the form $f(x)/(1+xg(x))$ for some polynomials $f,g\in\mathbb{Z}(x)$. This implies that the powers series expansions of the entries have integer coefficients. These coefficients stabilize during the algorithm and this provides a method to compute the power series of $\tau(G,x)$ around $0$. 

\begin{Ex} For the graph M\"obius ladder on $10$ vertices, a cycle on $10$ vertices together with the diagonal edges, the matrix $\A_G(x)$ at $x=0.3$ looks as follows.

\begin{minipage}{0.25\textwidth}
\begingroup%
\makeatletter%
\include{Mobius_ladder}%
\makeatother%
\endgroup%
\end{minipage}
\begin{minipage}{0.72\textwidth}
$$
\scalemath{0.64}{\left(\begin{array}{cccccccccc}
1 & 0.3 & 0.1128 & 0.0900 & 0.1598 & 0.3 & 0.1598 & 0.0900 & 0.1128 & 0.3 \\
0.3 & 1 & 0.3 & 0.1128 & 0.0900 & 0.1598 & 0.3 & 0.1598 & 0.0900 & 0.1128 \\
0.1128 & 0.3 & 1 & 0.3 & 0.1128 & 0.0900 & 0.1598 & 0.3 & 0.1598 & 0.0900 \\
0.0900 & 0.1128 & 0.3 & 1 & 0.3 & 0.1128 & 0.0900 & 0.1598 & 0.3 & 0.1598 \\
0.1598 & 0.0900& 0.1128 & 0.3 & 1 & 0.3 & 0.1128 & 0.0900 & 0.1598 & 0.3 \\
0.3 & 0.1598 & 0.0900 & 0.1128 & 0.3 & 1 & 0.3 & 0.1128 & 0.0900 & 0.1598 \\
0.1598 & 0.3 & 0.1598 & 0.0900 & 0.1128 & 0.3 & 1 & 0.3 & 0.1128 & 0.0900 \\
0.0900 & 0.1598 & 0.3 & 0.1598 & 0.0900 & 0.1128 & 0.3 & 1 & 0.3 & 0.1128 \\
0.1128 & 0.0900 & 0.1598 & 0.3 & 0.1598 & 0.0900 & 0.1128 & 0.3 & 1 & 0.3 \\
0.3 & 0.1128 & 0.0900 & 0.1598 & 0.3 & 0.1598 & 0.0900 & 0.1128 & 0.3 & 1
\end{array}\right)}
$$
\end{minipage}
\bigskip

\begin{center}
M\"obius ladder and $\A_G(x)$ at $x=0.3$.
\end{center}

\end{Ex}

\section{Sidorenko-type inequality} \label{operations}

In this section we prove Theorem~\ref{tau-main}, that is, we prove the inequality
$$\tau(G,x)\geq (1-x^2)^{e(G)}$$
for $x\in (0,1)$.
The proof relies on the understanding of the effect of two graph operations to the function $\tau(G,x)$, namely the graph operations edge deletion and edge contraction.

\begin{Th}\label{edge-deletion} For every edge $e=(u,v)\in E(G)$ we have
$$(1-(1-x^2)^2y_G(u,v)^2)\tau(G-e,x)\leq \tau(G,x).$$
In particular, if for some edge $e=(u,v)$ and some $x\in [0,1)$ we have $|y_G(u,v)|\leq \frac{x}{1-x^2}$, then
$$(1-x^2)\tau(G-e,x)\leq \tau(G,x).$$
\end{Th}

\begin{Rem} It is worth classifying the edges of a graph $G$ for a fixed $x\in (0,1)$ as follows. We distinguish three types:
\begin{itemize}
\item $e=(u,v)$ is of type I if $0\leq y_G(u,v)\leq \frac{x}{1-x^2}$,
\item $e=(u,v)$ is of type II if $y_G(u,v)<0$,
\item $e=(u,v)$ is of type III if $y_G(u,v)>\frac{x}{1-x^2}$.
\end{itemize}

The main consequence of Theorem~\ref{edge-deletion} is that the existence of one single type I edge for all pairs $(G,x)$ ensures an inductive proof of
$$\tau(G,x)\geq (1-x^2)^{e(G)}.$$
\end{Rem}

\begin{proof} Let us partition $\A_G(x)=\left(\begin{array}{cc} A_{11} & A_{12} \\ A_{21} & A_{22} \end{array}\right)$ such that $A_{11}$ corresponds to the $2\times 2$ matrix of $u$ and $v$ where $e=(u,v)$. Let $\B_G(x)=\left(\begin{array}{cc} B_{11} & B_{12} \\ B_{21} & B_{22} \end{array}\right)$ be the corresponding decomposition of the inverse matrix. Then by Theorem~\ref{Schur-complement} we have $A_{11}=(B_{11}-B_{12}B_{22}^{-1}B_{21})^{-1}$, that is, $A_{11}^{-1}=B_{11}-B_{12}B_{22}^{-1}B_{21}$.
Since $A_{11}=\left(\begin{array}{cc} 1 & x \\ x & 1\end{array}\right)$ we have 
$A_{11}^{-1}=\left(\begin{array}{cc} \frac{1}{1-x^2} & -\frac{x}{1-x^2} \\ -\frac{x}{1-x^2} & \frac{1}{1-x^2}\end{array}\right)$.

For an edge $(u,v)$ let $E_{u,v}(x)$ be the matrix which takes value $x$ at the entries $(u,u)$ and $(v,v)$ and $-1$ at the entries $(u,v)$ and $(v,u)$, and $0$ everywhere else. Let us consider the matrix
$$B(t)=\B_G(x)-tE_e(x).$$

First we show that $B(t)$ is positive definite for $t\in (-\frac{1}{1-x^2},\frac{1}{1-x^2})$. Indeed, by the last part of Theorem~\ref{Schur-complement} it is positive definite, if $B_{22}$ and its Schur complement 
$B_{11}+tE_e(x)-B_{12}B_{22}^{-1}B_{21}$ are both positive definite. The matrix $B_{22}$ is clearly positive definite since it is a submatrix of $\B_{G}(x)$. For the Schur complement we have
$$B_{11}-tE_e(x)-B_{12}B_{22}^{-1}B_{21}=A_{11}^{-1}-tE_e(x)=\left(\begin{array}{cc} \frac{1}{1-x^2}-tx & -\frac{x}{1-x^2}+t \\ -\frac{x}{1-x^2}+t & \frac{1}{1-x^2}-tx\end{array}\right).$$
This is positive definite if the determinants of the principal submatrices are positive, that is, $t<\frac{1}{x(1-x^2)}$ and the determinant itself is positive. The determinant is
\begin{align*}
\det(B_{11}-tE_e(x)-B_{12}B_{22}^{-1}B_{21})&=\left(\left(\frac{1}{1-x^2}-xt\right)^2-\left(-\frac{x}{1-x^2}+t\right)^2\right)\\
&=\frac{1}{1-x^2}\left(1-(1-x^2)^2t^2\right)
\end{align*}
This is also positive if $|t|<\frac{1}{1-x^2}$. Hence $B(t)$ is positive definite for $t\in (-\frac{1}{1-x^2},\frac{1}{1-x^2})$. In particular, $B(y_G(u,v))\in \mathcal{B}(G-e,x)$ if $|y_{G}(u,v)|<\frac{1}{1-x^2}$. If $|y_{G}(u,v)|\geq \frac{1}{1-x^2}$, then the statement of the theorem is trivial anyway.
Furthermore,
$$\det(B(t))=\det(B_{22}) \det(B_{11}-tE_e(x)-B_{12}B_{22}^{-1}B_{21}).$$
In case of $t=0$ we have $B(0)=\B_G(x)$ and so $\det(B_{22})=\frac{1-x^2}{\tau(G,x)}$.
Hence
$$\det(B(t))=\left(1-(1-x^2)^2t^2\right)\frac{1}{\tau(G,x)}.$$
In particular,
$$\frac{1}{\tau(G-e,x)}\geq \det(B(y_G(u,v)))=\left(1-(1-x^2)^2y_G(u,v)^2\right)\frac{1}{\tau(G,x)}.$$
In other words,
$$(1-(1-x^2)^2y_G(u,v)^2)\tau(G-e,x)\leq \tau(G,x).$$
When $|y_G(u,v)|\leq \frac{x}{1-x^2}$ we get that
$$(1-x^2)\tau(G-e,x)\leq \tau(G,x).$$
\end{proof}

Let $G/e$ be the graph obtained from $G$ by contracting the edge $e$ and deleting the possibly appearing multiple edges. Then $|E(G/e)|\leq e(G)-1$.

\begin{Th} \label{edge-contraction} Suppose that for some graph $G$, an edge $e=(u,v)$ and some $x\in [0,1)$ we have $y_G(u,v)\geq \frac{x}{1-x^2}$. Then
$$\tau(G,x)\geq (1-x^2)\tau(G/e,x).$$
\end{Th}

\begin{Rem} 
Theorem~\ref{edge-contraction} implies that a type III edge also ensures the inductive proof of 
$$\tau(G,x)\geq (1-x^2)^{e(G)}.$$
We will see that there is always a type I or type III edge.
\end{Rem}

\begin{proof}
As before let us partition $\A_G(x)=\left(\begin{array}{cc} A_{11} & A_{12} \\ A_{21} & A_{22} \end{array}\right)$ such that $A_{11}$ corresponds to the $2\times 2$ matrix of $u$ and $v$ where $e=(u,v)$. Let $\B_G(x)=\left(\begin{array}{cc} B_{11} & B_{12} \\ B_{21} & B_{22} \end{array}\right)$ be the corresponding decomposition of the inverse matrix. Then $A_{11}=(B_{11}-B_{12}B_{22}^{-1}B_{21})^{-1}$, that is, $A_{11}^{-1}=B_{11}-B_{12}B_{22}^{-1}B_{21}$.
Since $A_{11}=\left(\begin{array}{cc} 1 & x \\ x & 1\end{array}\right)$ we have 
$A_{11}^{-1}=\left(\begin{array}{cc} \frac{1}{1-x^2} & -\frac{x}{1-x^2} \\ -\frac{x}{1-x^2} & \frac{1}{1-x^2}\end{array}\right)$.
Let $-\underline{y_u}$ and $-\underline{y_v}$ be the two column vectors of $B_{21}$, so these vectors contain the entries $(y_G(u,w))_w$ and  $(y_G(v,w))_w$ for $w\in V(G)\setminus \{u,v\}$. Then we obtain the following equations by comparing $A_{11}^{-1}=B_{11}-B_{12}B_{22}^{-1}B_{21}$:
$$1+xy_G(u,v)+x\sum_{w}y_G(u,w)-\underline{y_u}^TB_{22}^{-1}\underline{y_u}=\frac{1}{1-x^2},$$
$$1+xy_G(v,u)+x\sum_{w}y_G(v,w)-\underline{y_v}^TB_{22}^{-1}\underline{y_v}=\frac{1}{1-x^2},$$
and
$$-y_G(u,v)-\underline{y_u}^TB_{22}^{-1}\underline{y_v}=-\frac{x}{1-x^2}.$$
\medskip

Now let us consider the following matrix $B$ corresponding to $G/e$. Let $s$ be the new vertex that we get by contracting $u$ and $v$. If $w_1,w_2\neq s$, then set $B_{w_1,w_2}=\B_G(x)_{w_1,w_2}$. For $w\neq s$ let
$$t(s,w)=-B_{s,w}=-\B_G(x)_{w,u}-\B_G(x)_{w,v}=y_G(w,u)+y_G(w,v),$$
and
$$B_{s,s}=1+x\sum_{w}y_G(u,w)+x\sum_{w}y_G(v,w).$$
We will show that the matrix $B$ is positive definite. This will imply that $B\in \mathcal{B}(G/e,x)$. Since the matrix $B_{22}$  is  a principal submatrix of $\B_G(x)$, it is positive definite. Thus  we only need to show that $\det(B)>0$ by Sylvester's criterion (Theorem~\ref{Sylvester}). Note that
$$\det(B)=\det(B_{22})\left(1+x\sum_{w}y_G(u,w)+x\sum_{w}y_G(v,w)-(\underline{y_u}+\underline{y_v})^TB_{22}^{-1}(\underline{y_u}+\underline{y_v})\right).$$
Here $\det(B_{22})=\frac{1-x^2}{\tau(G,x)}$.
Furthermore, using the above equations we get that
$$1+x\sum_{w}y_G(u,w)+x\sum_{w}y_G(v,w)-(\underline{y_u}+\underline{y_v})^TB_{22}^{-1}(\underline{y_u}+\underline{y_v})$$
$$1+\left(\frac{1}{1-x^2}-1-xy_G(u,w)\right)+\left(\frac{1}{1-x^2}-1-xy_G(v,w)\right)+2\left(y_G(u,v)-\frac{x}{1-x^2}\right)=$$
$$=\frac{1-2x+x^2}{1-x^2}+2(1-x)y_G(u,v)=\frac{1-x}{1+x}+2(1-x)y_G(u,v)$$
Hence
$$\det(B)=\left(\frac{1-x}{1+x}+2(1-x)y_G(u,v)\right)\frac{1-x^2}{\tau(G,x)}.$$
This is clearly positive if $y_G(u,v)\geq 0$ so $B\in \mathcal{B}(G/e,x)$.
Furthermore, if $y_G(u,v)\geq \frac{x}{1-x^2}$, then
$$\frac{1-x}{1+x}+2(1-x)y_G(u,v)\geq 1$$
so
$$\frac{1}{\tau(G/e,x)}=\max_{B'\in \mathcal{B}(G/e,x)}\det(B')\geq \det B\geq \frac{1-x^2}{\tau(G,x)}.$$
Hence
$$\tau(G,x)\geq (1-x^2)\tau(G/e,x).$$
\end{proof}

\begin{Rem} We honestly confess that we have never seen a graph and a type III edge $e=(u,v)$, that is, an edge for which $y_G(u,v)>\frac{x}{1-x^2}$, and if $y_G(u,v)=\frac{x}{1-x^2}$, then $e$ was a cut edge. It can be shown that for a cut edge $e$ we always have $y_G(u,v)=\frac{x}{1-x^2}$.
\end{Rem}

\begin{proof}[Proof of Theorem~\ref{tau-main}.] We prove the statement by induction on the number of edges. If the graph has no edges, then the claim is trivial. First we show that there is always an edge $e=(u,v)$ for which $y_G(u,v)\geq 0$. In fact, for any vertex $u$ we have
$$\sum_{v\in N_G(u)}y_G(u,v)\geq 0.$$
An easy way to see this is the following: as before let us partition $\A_G(x)=\left(\begin{array}{cc} A_{11} & A_{12} \\ A_{21} & A_{22} \end{array}\right)$ such that $A_{11}$ corresponds to the  vertex $u$. Let $\B_G(x)=\left(\begin{array}{cc} B_{11} & B_{12} \\ B_{21} & B_{22} \end{array}\right)$ be the corresponding decomposition of the inverse matrix. Then $A_{11}=(B_{11}-B_{12}B_{22}^{-1}B_{21})^{-1}$, that is, $A_{11}^{-1}=B_{11}-B_{12}B_{22}^{-1}B_{21}$. Note that $A_{11}^{-1}=1$, and $B_{22}^{-1}$ is positive definite, and $B_{21}$ is just the transpose of the vector $B_{12}$. Thus $1=A_{11}^{-1}=B_{11}-B_{12}B_{22}^{-1}B_{21}\leq B_{11}$. Since $B_{11}=1+x\sum_{v\in N_G(u)}y_G(u,v)$ we immediately get that $\sum_{v\in N_G(u)}y_G(u,v)\geq 0$.

Now suppose that for some edge $e=(u,v)$ we have $y_G(u,v)\geq 0$. If $y_G(u,v)\leq \frac{x}{1-x^2}$, then by Theorem~\ref{edge-deletion} we have $\tau(G,x)\geq (1-x^2)\tau(G-e,x)$.
By induction
$$\tau(G,x)\geq (1-x^2)\tau(G-e,x)\geq (1-x^2)(1-x^2)^{|E(G-e)|}=(1-x^2)^{e(G)}.$$
If $y_G(u,v)\geq \frac{x}{1-x^2}$, then by Theorem~\ref{edge-contraction} we have
$\tau(G,x)\geq (1-x^2)\tau(G/e,x)$. By induction
$$\tau(G,x)\geq (1-x^2)\tau(G/e,x)\geq (1-x^2)(1-x^2)^{|E(G/e)|}\geq (1-x^2)^{e(G)}.$$
We are done.
\end{proof}

\begin{proof}[Proof of Theorem~\ref{main2}.]
This theorem immediately follows from Theorem~\ref{tau-main} and Lemma~\ref{bipartite-symmetry}.
\end{proof}

We end this section with a counterpart of Theorem~\ref{edge-deletion}. We do not prove this statement as its proof strongly follows the proofs of Theorems~\ref{edge-deletion} and \ref{edge-contraction}.

\begin{Th}
Let $G$ be a graph and let $e=(u,v)\in E(G)$. Let $z=z_{G-e}(u,v)$. Then
$$\tau(G-e,x)\geq \frac{(1-zx)^2}{(1-z^2)(1-x^2)}\tau(G,x).$$
In particular, if $z\geq \frac{2x}{1-x^2}$ or $z\leq 0$, then
$$(1-x^2)\tau(G-e,x)\geq \tau(G,x).$$
\end{Th}

\section{Logarithmic derivative} \label{second-interval}

In this section we study the logarithmic derivative of the function $\tau(G,x)$. We prove Theorems~\ref{log-ineq} and \ref{tightness} for the so-called second interval, that is, the interval $[\frac{1}{\overline{d}-1},1]$. The identity of Lemma~\ref{log-derivative} is the main statement in this section. 

\begin{Lemma} \label{log-derivative} We have
$$-\frac{\tau'(G,x)}{\tau(G,x)}=2\sum_{(u,v)\in E(G)}y_G(u,v).$$
\end{Lemma}

\begin{proof}
Suppose $A(x)$ is a matrix with entries $a_{ij}(x)$.
Then
$$\frac{d}{dx}\det(A(x))=\frac{d}{dx}\left(\sum_{\pi \in S_n}(-1)^{\mathrm{sign}(\pi)}\prod_{i=1}^na_{i,\pi(i)}(x)\right)=\sum_{\pi \in S_n}(-1)^{\mathrm{sign}(\pi)}\frac{d}{dx}\left(\prod_{i=1}^na_{i,\pi(i)}(x)\right)=$$
$$=\sum_{i,j}a_{ij}'(x)(-1)^{i+j}\det(A^{ij}(x)),$$
where $A^{ij}(x)$ is the matrix which we obtain from $A(x)$ by deleting the $i$.th row and $j$.th column. Clearly,
$$(-1)^{i+j}\det(A^{ij}(x))=\det(A(x))\cdot A^{-1}_{ji}(x).$$
Hence
$$\frac{\frac{d}{dx}\det(A(x))}{\det(A(x))}=\sum_{i,j}a_{ij}'(x)\cdot A^{-1}_{ji}(x).$$
Let us apply this to $\A_G(x)$: if the vertices $i$ and $j$ are adjacent, then $\A_G(x)_{ij}=x$ so its derivative is $1$, and $\A_G(x)^{-1}_{ij}=\B_G(x)_{ij}=-y_G(i,j)$. If the vertices $i$ and $j$ are distinct non-adjacent vertices, then $\A^{-1}_G(x)_{ij}=\B_G(x)_{ij}=0$. Finally if $i=j$, then the derivative of $\A_G(x)_{ii}=1$ is $0$. Hence
$$-\frac{\tau'(G,x)}{\tau(G,x)}=2\sum_{(u,v)\in E(G)}y_G(u,v).$$
\end{proof}

\begin{Lemma} \label{log-derivative_upper_bound} We have
$$2\sum_{e\in E(G)}y_G(u,v)\leq \frac{n-1}{1-x}.$$
\end{Lemma}

\begin{proof}
First we prove the slightly weaker result
$$2\sum_{e\in E(G)}y_G(u,v)\leq \frac{n}{1-x}.$$
The crucial observation is that since $\B_G(x)$ is positive definite we have $\underline{1}^T\B_G(x)\underline{1}> 0$, where $\underline{1}$ is the all-$1$ vector. Clearly,
$$0<\underline{1}^T\B_G(x)\underline{1}=n+2(x-1)\sum_{(u,v)\in E(G)}y_G(u,v).$$
Hence
$$2\sum_{(u,v)\in E(G)}y_G(u,v)< \frac{n}{1-x}.$$
To improve on this result we observe that $\underline{1}^T\B_G(x)\underline{1}$ cannot be arbitrarily small. Indeed, by Theorem~\ref{quadratic-pos-def} we know that for any positive definite matrix $A$ and any vector $\underline{x}$ we have
$$\underline{x}^TA^{-1}\underline{x}\cdot \underline{x}^TA\underline{x}\geq ||\underline{x}||^4.$$
Applying this result to $\B_G(x)$ and $\underline{1}$ we get that
$$\underline{1}^T\B_G(x)\underline{1} \cdot \underline{1}^T\A_G(x)\underline{1}\geq ||\underline{1}||^4=n^2.$$
Note that $\underline{1}^T\A_G(x)\underline{1}\leq n^2$ as any element of $\A_G(x)$ is at most $1$. Consequently, $\underline{1}^T\B_G(x)\underline{1}\geq 1$, and
$$1\leq \underline{1}^T\B_G(x)\underline{1}=n+2(x-1)\sum_{(u,v)\in E(G)}y_G(u,v).$$
Hence
$$2\sum_{(u,v)\in E(G)}y_G(u,v)\leq \frac{n-1}{1-x}.$$
\end{proof}

\begin{Rem} By applying the inequality 
$$\underline{x}^TA^{-1}\underline{x}\cdot \underline{x}^TA\underline{x}\geq ||\underline{x}||^4.$$
to the matrix $\B_G(x)$ and the characteristic vector $e_S$ that takes $1$ at the vertices of $S$, and $0$ everywhere else we get that
$$(|S|-1)+x\sum_{e\in E(S,V\setminus S)}y_e\geq 2(1-x)\sum_{e\in E(S)}y_e.$$
If $S=V(G)$ we get the above lemma. If $S=\{u\}$, then we get that $\sum_{v\in N_G(u)}y_G(u,v)\geq 0$, an inequality that we used in the proof of Theorem~\ref{tau-main}.
\end{Rem}

\begin{Rem}
In the next few applications the inequality 
$$\frac{\tau'(G,x)}{\tau(G,x)}=-2\sum_{(u,v)\in E(G)}y_G(u,v)\geq -\frac{n}{1-x}$$
will be sufficient for us.
\end{Rem}

Now we are ready to give the proofs of Theorems~\ref{log-ineq} and \ref{tightness} for the interval $[\frac{1}{\overline{d}-1},1]$.

\begin{proof}[Proof of Theorem~\ref{log-ineq} for the second interval.]
From the previous two lemmas we know that
$$-\frac{\tau'(G,x)}{\tau(G,x)}=2\sum_{(u,v)\in E(G)}y_G(u,v)\leq \frac{n}{1-x}.$$
It is easy to check that
$$\frac{n}{1-x} \leq 2e(G)\frac{x}{1-x^2}$$
if $x\geq \frac{1}{\overline{d}-1}$ as required.
\end{proof}

\begin{proof}[Proof of Theorem~\ref{tightness} for the second interval.]
Set $u=\frac{1}{\overline{d}-1}$
We have
\begin{align*}
\ln \tau(G,x)&=\ln \tau(G,u)+\int_{u}^{x} \frac{\tau'(G,t)}{\tau(G,t)}dt \\
             &\geq e(G)\ln (1-u^2)+\int_{u}^{x} -\frac{n}{1-t}dt \\
             &=e(G)\ln (1-u^2)+e(G)\int_{u}^x\frac{-2t}{1-t^2}dt+\int_{u}^{x}\left(\frac{-n}{1-t}+\frac{2e(G)t}{1-t^2}\right)dt\\
						&=e(G)\ln(1-x^2)+e(G)\int_{u}^{x}\left(-\frac{2}{\overline{d}(1-t)}+\frac{2t}{1-t^2}\right)dt\\
						&=e(G)\ln(1-x^2)+e(G)\int_{u}^{x}\frac{2(t(\overline{d}-1)-1)}{\overline{d}(1-t^2)}dt
\end{align*}
Hence
$$\frac{\ln \tau(G,x)}{e(G)}-\ln(1-x^2)\geq \int_{u}^{x}\frac{2(t(\overline{d}-1)-1)}{\overline{d}(1-t^2)}dt.$$
Set
$$\alpha(\overline{d},x)=\int_{u}^{x}\frac{2(t(\overline{d}-1)-1)}{\overline{d}(1-t^2)}dt.$$
This function is clearly positive as the integrand is positive.
\end{proof}

\section{Totally positive of order $2$} \label{M-graphs}

In this section we prove Theorem~\ref{log-ineq}  for the so-called first interval, that is, the interval $[0,\frac{1}{\Delta-1}]$. (The proof of Theorem~\ref{tightness} for the first interval will be given in Section~\ref{Ihara-zeta_section}.) We also prove Theorem~\ref{regular-complete}.

We say that a probability distribution with density function $f:\mathbb{R}^n\to \mathbb{R}$ is totally positive of order $2$ if it satisfies the inequality 
$$f(\underline{u})f(\underline{v})\leq f(\underline{u}\vee \underline{v})f(\underline{u}\wedge \underline{v})$$
for all $\underline{u}=(u_1,\dots ,u_n)$ and $\underline{v}=(v_1,\dots ,v_n)$, where
$\underline{u}\vee \underline{v}$ and $\underline{u}\wedge \underline{v}$ are the vectors with coordinates $(\underline{u}\vee \underline{v})_i=\max(u_i,v_i)$ and $(\underline{u}\wedge \underline{v})_i=\min(u_i,v_i)$.
To a centered multivariate Gaussian distribution with covariance matrix $A$, the density function
$$\frac{1}{(2\pi)^{n/2}\det(A)^{1/2}}\exp\left(-\frac{1}{2}\underline{u}A^{-1}\underline{u}\right)$$
is totally positive of order $2$ if and only if all off-diagonal elements of $A^{-1}$ are non-positive \cite{Bol,KR1}. Such matrices are called $M$-matrices.

\begin{Def}
A matrix is called an $M$-matrix if all off-diagonal entries are non-positive and its eigenvalues have non-negative real parts. 
\end{Def} 

There are many equivalent characterizations of M-matrices, for details see Chapter 2.5 of \cite{HJ2}. It is known for instance that if $B$ is an invertible $M$-matrix, then its inverse has only non-negative entries, see Theorem~\ref{M-matrix-inverse} in the Appendix. In fact, for non-singular matrices with non-positive off-diagonal elements this is an equivalent characterization.  In case of a symmetric matrix $B$, the matrix $B$ is an $M$-matrix if and only if the matrix is positive semidefinite and all off-diagonal entries are non-positive.

Let us consider the matrix $\B_G(x)=\A_G(x)^{-1}$. The off-diagonal elements of $\B_G(x)$ are $-y_G(u,v)$ for the edges $(u,v)\in E(G)$. Hence this is an $M$-matrix if and only if $y_G(u,v)\geq 0$ for all $(u,v)\in E(G)$. To sum up, the homogeneous Gaussian Markov random field for the pair $(G,x)$ is totally positive of order $2$ if and only if $y_G(u,v)\geq 0$ for all edges $(u,v)$. Then it is an interesting question in its own right that for which pairs $(G,x)$ satisfy that  $y_G(u,v)\geq 0$ for all edges $(u,v)$.  
In this direction we prove two theorems.

\begin{Th} \label{B is M-matrix} Let $G$ be a graph with largest degree $\Delta$. If $0\leq x< \frac{1}{\Delta-1}$, then the matrix $\B_G(x)$ is an $M$-matrix.
\end{Th}

\begin{Th} \label{vertex-transitive}
Let $G$ be a vertex-transitive graph and $x\in (0,1)$. Then for every $(u,v)\in E(G)$ we have $y_G(u,v)>0$, and for every $(u,v)\notin E(G)$ we have $z_G(u,v)<x$.
\end{Th}

Recall that we have  classified the edges of a graph $G$ for a fixed $x\in (0,1)$ as follows. We distinguish three types:
\begin{itemize}
\item $e=(u,v)$ is of type I if $0\leq y_G(u,v)\leq \frac{x}{1-x^2}$,
\item $e=(u,v)$ is of type II if $y_G(u,v)<0$,
\item $e=(u,v)$ is of type III if $y_G(u,v)>\frac{x}{1-x^2}$.
\end{itemize}

Thus $\B_G(x)$ is an $M$-matrix if $y_G(u,v)\geq 0$ for all $(u,v)\in E(G)$, that is, all edges are of type I or III. 

\begin{Rem} Computer simulations suggest that for a random graph all edges are of type I. It is possible to construct a graph with a type II edge. We have never seen a type III edge.

Theorem~\ref{edge-deletion} shows that if $e$ is of type I, then $\tau(G,x)\geq (1-x^2)\tau(G-e,x)$, while Theorem~\ref{edge-contraction} shows that if $e$ is of type III, then $\tau(G,x)\geq (1-x^2)\tau(G/e,x)$. A consequence of Lemma~\ref{log-derivative} is the following: if there is no type III edge, then
$$-\frac{\tau'(G,x)}{\tau(G,x)}=2\sum_{(u,v)\in E(G)}y_G(u,v)\leq 2e(G)\frac{x}{1-x^2},$$
and after integration and multiplication by $-1$ we get that $\ln \tau(G,x)\geq e(G)\ln(1-x^2)$, or equivalently $\tau(G,x)\geq (1-x^2)^{e(G)}$.

Below we will show that if there is no type II edge, then all edges are of type I, see Theorem~\ref{M-y}. Furthermore, if $x\in (0,\frac{1}{\Delta-1})$, then all edges are of type I. We will also show that if $G$ is a vertex-transitive graph, then all edges are of type I for all $x\in (0,1)$, see Theorem~\ref{vertex-transitive}.
\end{Rem}

In this section we study graphs with only edges of type I. We will utilize the classical theory of $M$-matrices that is widely studied in matrix analysis.

\begin{Th} \label{M-y}
Suppose that for some graph $G$ and some $x\in [0,1)$ the matrix $\B_G(x)$ is an $M$-matrix. Furthermore, suppose that the edge $(u,v)$ is in a clique $K_r$. Then $y_G(u,v)\leq \frac{x}{(1-x)(1+(r-1)x)}$. In particular, we have $y_G(u,v)\leq \frac{x}{1-x^2}$ for all edges $(u,v)$ in this case.
\end{Th}

\begin{Rem} In words, Theorem~\ref{M-y} asserts that if there are no edges of type II, then all edges are actually  of type I.  
\end{Rem}

\begin{proof}
Let us partition $\A_G(x)=\left(\begin{array}{cc} A_{11} & A_{12} \\ A_{21} & A_{22} \end{array}\right)$ such that $A_{11}$ corresponds to the clique of size $r$ containing the edge $(u,v)$. Let $\B_G(x)=\left(\begin{array}{cc} B_{11} & B_{12} \\ B_{21} & B_{22} \end{array}\right)$ be the corresponding decomposition of the inverse matrix. Then
$A_{11}=(B_{11}-B_{12}B_{22}^{-1}B_{21})^{-1}$  or in other words,
$B_{11}-A_{11}^{-1}=(-B_{12})B_{22}^{-1}(-B_{21})$. Here $B_{22}$ is an invertible $M$-matrix since it is positive definite and $\B_G(x)$ has non-positive off-diagonal elements. Hence the matrices $-B_{12},B_{22}^{-1},-B_{21}$ are all non-negative as $\B_G(x)$ is an $M$-matrix. Hence all elements of $B_{11}$ is larger than the corresponding element of $A_{11}^{-1}$. The off-diagonal elements of $A_{11}^{-1}$ are all $-\frac{x}{(1-x)(1+(r-1)x)}$. Hence
$$-y_G(u,v)\geq -\frac{x}{(1-x)(1+(r-1)x)},$$
equivalently,
$$y_G(u,v)\leq \frac{x}{(1-x)(1+(r-1)x)}.$$
Clearly, the inequality
$$y_G(u,v)\leq \frac{x}{1-x^2}$$
follows from the fact that an edge is in a $K_2$ so we can apply the claim to $r=2$.
\end{proof}

\begin{Rem}
It is also possible to prove that if $\B_G(x)$ is an M-matrix, then $z_G(u,v)\geq x^{\mathrm{dist}(u,v)}$ for any vertices $u$ and $v$. Note that for trees we have equality in the bounds $y_G(u,v)\leq \frac{x}{1-x^2}$ and $z_G(u,v)\geq x^{\mathrm{dist}(u,v)}$.
\end{Rem}

\begin{Def} Let
$$M(G)=\sup \{x \ |\ \B_G(t) \ \ \mbox{is an $M$-matrix for $t\in [0,x]$}\}.$$
We say that a graph $G$ is an $M$-graph if $M(G)=1$.
\end{Def}

Clearly, $M(G)\geq 0$ since $\B_G(0)$ is an $M$-matrix. We will show that $M(G)\geq \frac{1}{\Delta-1}$, where $\Delta$ is the largest degree of $G$.

\begin{Cor} \label{M-ineq} For $x\in [0,M(G)]$ we have
$$\frac{\tau'(G,x)}{\tau(G,x)}\geq e(G)\frac{\tau'(K_2,x)}{\tau(K_2,x)}.$$
\end{Cor}

\begin{proof} This corollary immediately follows from Lemma~\ref{log-derivative} and the fact that for $x\in [0,M(G))$ we have $y_G(u,v)\leq \frac{x}{1-x^2}$ for adjacent vertices $u$ and $v$. Indeed, 
$$-\frac{\tau'(G,x)}{\tau(G,x)}=2\sum_{(u,v)\in E(G)}y_G(u,v)\leq 2e(G)\frac{x}{1-x^2}.$$
For the point $x=M(G)$ the claim follows from continuity.
\end{proof}

Our next goal is to prove that for small enough $x$ the matrix $\B_G(x)$ is indeed an $M$-matrix, see Theorem~\ref{B is M-matrix}. We need some preparation.

\begin{Lemma} \label{small-x-M-matrix} Suppose that for some graph $G$ and some $x\in (0,1)$ the matrix $\B_G(x)$ is an $M$-matrix. Furthermore, suppose that $0<x< \frac{1}{\Delta-1}$, where $\Delta$ is the largest degree. Then  we have \\
(a) $\B_G(x)$ is diagonally dominant,\\
(b) $z_G(u,v)<x$ for all $(u,v)\notin E(G)$.
\end{Lemma}

\begin{proof} (a) By Theorem~\ref{M-y} we have $y_G(u,v)\leq \frac{x}{1-x^2}$ for adjacent vertices $u$ and $v$. The matrix $\B_G(x)$ is diagonally dominant if 
$$1+x\sum_{v\in N_G(u)}y_G(u,v)>\sum_{v\in N_G(u)}y_G(u,v).$$
This is satisfied since
$$(1-x)\sum_{v\in N_G(u)}y_G(u,v)\leq (1-x)\Delta \frac{x}{1-x^2}=\frac{\Delta x}{1+x}<1$$
if $x< \frac{1}{\Delta-1}$.
\medskip

(b) Let $z^*=\max_{u \neq v}z_G(u,v)$. Note that $z^*>0$ since $x>0$. We show that $z^*$ is only achieved on adjacent vertices, and consequently its value is $x$. Suppose for contradiction that $z^*$ is achieved for some $(u,w)\notin E(G)$. Then 
$$z_G(w,u)=\sum_{v\in N_G(w)}z_G(v,u)\frac{y_G(v,w)}{1+x\sum_{r\in N_G(w)}y_G(r,w)}.$$
Then
\begin{align*}
z^*=z_G(w,u)&=\sum_{v\in N_G(w)}z_G(v,u)\frac{y_G(v,w)}{1+x\sum_{r\in N_G(w)}y_G(r,w)} \\
&\leq z^*\sum_{v\in N_G(w)}\frac{y_G(v,w)}{1+x\sum_{r\in N_G(w)}y_G(r,w)}<z^*\\
\end{align*}
by part (a) which is a contradiction.
\end{proof}

\begin{Lemma} Let $G$ be a graph and $e=(u,v)\in E(G)$. Then $y_{G}(u,v)=0$ if and only if $z_{G-e}(u,v)=x$.
\end{Lemma}

\begin{proof} First suppose that $z_{G-e}(u,v)=x$. Note that $\tau(G-e,x)\geq \tau(G,x)$ is always true. If $z_{G-e}(u,v)=x$, then the matrix $\A_{G-e}(x)$ satisfies the conditions to be in $\mathcal{A}(G,x)$ so $\tau(G,x)\geq \tau(G-e,x)$. Thus $\tau(G,x)=\tau(G-e,x)$ and since the maximizer is unique we have $\A_G(x)=\A_{G-e}(x)$. But then $\B_G(x)=\B_{G-e}(x)$ implying that $y_G(u,v)=0$.

Next suppose that $y_{G}(u,v)=0$. This time we use Lemma~\ref{dual}: $$\max_{B\in \mathcal{B}(G,x)}\det(B)\geq \max_{B\in \mathcal{B}(G-e,x)}\det(B)$$ is always true, and if for the maximizing matrix $\B_G(x)\in \mathcal{B}(G,x)$ we have $y_{G}(u,v)=0$, then the opposite inequality is also true. Since the maximizing matrix is unique we get that they are equal, consequently, for their inverses we have $\A_G(x)=\A_{G-e}(x)$ implying that $z_{G-e}(u,v)=x$.
\end{proof}

Now we are ready to prove Theorem~\ref{B is M-matrix}.
In plain words, this theorem says that if $x$ is small, that is, $x\in (0,\frac{1}{\Delta-1})$, then all edges are of type I.

\begin{proof}[Proof of Theorem~\ref{B is M-matrix}] We will show that $y_G(u,v)>0$ for every $x$ in the interval $(0,\frac{1}{\Delta-1})$ for all $(u,v)\in E(G)$. Suppose for contradiction that it is not true and consider a counterexample with smallest possible number of edges.

It is not hard to see that $y_G(u,v)=x+O(x^2)$ for all $(u,v)\in E(G)$. This means that for every small enough positive $x$ we have $y_G(u,v)>0$. If for some $x$ in the interval $(0,\frac{1}{\Delta-1})$ and for some $(u,v)\in E(G)$ we have $y_G(u,v)<0$, then by the continuity of $y_G(u,v)$ we know that there must be an $x$ in this interval where $y_G(u,v)=0$. Then $z_{G-e}(u,v)=x$. On the other hand, $G-e$ has fewer edges and $\Delta(G-e)\leq \Delta(G)$ so $\B_{G-e}(x)$ is an $M$-matrix by the assumption on $G$ being the smallest counterexample. Then  $0<x< \frac{1}{\Delta-1}$ and  the fact that $\B_{G-e}(x)$ is an $M$-matrix implies that $z_{G-e}(u,v)<x$ by Lemma~\ref{small-x-M-matrix}, contradiction.
\end{proof}

Now we are ready to prove Theorem~\ref{log-ineq} for the interval $[0,\frac{1}{\Delta-1}]$.

\begin{proof}[Proof of Theorem~\ref{log-ineq} for the first interval.]
This is now trivial from Corollary~\ref{M-ineq} and Theorem~\ref{B is M-matrix}. 
\end{proof}

\subsection{Regular and vertex-transitive graphs} In this section we study regular and vertex-transitive graphs. In particular, we prove Theorems~\ref{regular-complete} and ~\ref{vertex-transitive}.

Let
$$Y_u=\sum_{v\in N_G(u)}y_G(u,v).$$

\begin{Lemma} \label{Schur complement bounds}
Let $G$ be an arbitrary graph and $x\in (0,1)$. Then
\medskip

\noindent (a) If $u$ is not an isolated vertex, then $Y_u\geq \frac{x}{1-x^2}$. \\
\noindent (b) If $(u,v)\in E(G)$, then $\frac{2}{1-x}\leq (1+xY_u)+(1+xY_v)+2y_G(u,v)$.\\
\noindent (c) If $(u,v)\notin E(G)$, then $\frac{2}{1-z_G(u,v)}\leq (1+xY_u)+(1+xY_v)$.
\end{Lemma}

\begin{proof}
We first we prove part (a) and (b). Let us partition $\A_G(x)=\left(\begin{array}{cc} A_{11} & A_{12} \\ A_{21} & A_{22} \end{array}\right)$ such that $A_{11}$ corresponds to the $2\times 2$ matrix of $u$ and $v$ where $e=(u,v)$. (In part (a), simply choose any neighbor of the vertex $u$ and let it be $v$.) Let $\B_G(x)=\left(\begin{array}{cc} B_{11} & B_{12} \\ B_{21} & B_{22} \end{array}\right)$ be the corresponding decomposition of the inverse matrix. Then $A_{11}^{-1}=B_{11}-B_{12}B_{22}^{-1}B_{21}$.
Since $A_{11}=\left(\begin{array}{cc} 1 & x \\ x & 1\end{array}\right)$ we have 
$A_{11}^{-1}=\left(\begin{array}{cc} \frac{1}{1-x^2} & -\frac{x}{1-x^2} \\ -\frac{x}{1-x^2} & \frac{1}{1-x^2}\end{array}\right)$.
Let $-\underline{y_u}$ and $-\underline{y_v}$ be the two column vectors of $B_{21}$, so the vectors $\underline{y_u},\underline{y_v}$ contain the entries $(y_G(u,w))_w$ and  $(y_G(v,w))_w$ for $w\in V(G)\setminus \{u,v\}$. Then we obtain the following equations by comparing $A_{11}^{-1}=B_{11}-B_{12}B_{22}^{-1}B_{21}$:
$$1+xY_u-\underline{y_u}^TB_{22}^{-1}\underline{y_u}=\frac{1}{1-x^2},\ \ 1+xY_v-\underline{y_v}^TB_{22}^{-1}\underline{y_v}=\frac{1}{1-x^2},\ \ -y_G(u,v)-\underline{y_u}^TB_{22}^{-1}\underline{y_v}=-\frac{x}{1-x^2}.$$
From the first equation and the fact that $B_{22}^{-1}$ is positive definite we immediately get that
$$1+xY_u\geq 1+xY_u-\underline{y_u}^TB_{22}^{-1}\underline{y_u}=\frac{1}{1-x^2}$$
implying that  $Y_u\geq \frac{x}{1-x^2}$. This proves part (a). To prove part (b) observe that
\begin{align*}
0&\leq (\underline{y_u}-\underline{y_v})^TB_{22}^{-1}(\underline{y_u}-\underline{y_v})\\
 &=(1+xY_u)-\frac{1}{1-x^2}+(1+xY_v)-\frac{1}{1-x^2}+2\left(y_G(u,v)-\frac{x}{1-x^2}\right)\\
 &=(1+xY_u)+(1+xY_v)+2y_G(u,v)-\frac{2}{1-x}
\end{align*}
This proves part (b). 
\bigskip

The proof of part (c) is completely analogous. Set $z=z_G(u,v)$. We partition $\A_G(x)=\left(\begin{array}{cc} A_{11} & A_{12} \\ A_{21} & A_{22} \end{array}\right)$ such that $A_{11}$ corresponds to the $2\times 2$ matrix of $u$ and $v$ where $(u,v)\notin E(G)$ this time. Let $\B_G(x)=\left(\begin{array}{cc} B_{11} & B_{12} \\ B_{21} & B_{22} \end{array}\right)$ be the corresponding decomposition of the inverse matrix. Then
$A_{11}^{-1}=B_{11}-B_{12}B_{22}^{-1}B_{21}$.
Since $A_{11}=\left(\begin{array}{cc} 1 & z \\ z & 1\end{array}\right)$ we have 
$A_{11}^{-1}=\left(\begin{array}{cc} \frac{1}{1-z^2} & -\frac{z}{1-z^2} \\ -\frac{z}{1-z^2} & \frac{1}{1-z^2}\end{array}\right)$.
Let $-\underline{y_u}$ and $-\underline{y_v}$ be the two column vectors of $B_{21}$, so the vectors $\underline{y_u},\underline{y_v}$ contain the entries $(y_G(u,w))_w$ and  $(y_G(v,w))_w$ for $w\in V(G)\setminus \{u,v\}$. Then we obtain the following equations by comparing $A_{11}^{-1}=B_{11}-B_{12}B_{22}^{-1}B_{21}$:
$$1+xY_u-\underline{y_u}^TB_{22}^{-1}\underline{y_u}=\frac{1}{1-z^2},\ \ 1+xY_v-\underline{y_v}^TB_{22}^{-1}\underline{y_v}=\frac{1}{1-z^2},\ \ -\underline{y_u}^TB_{22}^{-1}\underline{y_v}=-\frac{z}{1-z^2}.$$
Now observe that
\begin{align*}
0&\leq (\underline{y_u}-\underline{y_v})^TB_{22}^{-1}(\underline{y_u}-\underline{y_v})\\
 &=(1+xY_u)-\frac{1}{1-z^2}+(1+xY_v)-\frac{1}{1-z^2}-2\frac{z}{1-z^2} \\
 &=(1+xY_u)+(1+xY_v)-\frac{2}{1-z}
\end{align*}
This proves part (c). 
\end{proof}

Now we are ready to prove Theorem~\ref{regular-complete}.

\begin{proof}[Proof of Theorem~\ref{regular-complete}]
Let $n$ denote the number of vertices of $G$. By part (a) of the lemma we have 
$$\frac{2}{1-x}\leq (1+xY_u)+(1+xY_v)+2y_G(u,v).$$
for every $(u,v)\in E(G)$. By summing this for all edges we get that
$$\frac{nd}{1-x}\leq \sum_{u\in V(G)}d(1+xY_u)+2\sum_{(u,v)\in E(G)}y_G(u,v)=nd+2(1+dx)\sum_{(u,v)\in E(G)}y_G(u,v).$$
Hence 
$$2\sum_{(u,v)\in E(G)}y_G(u,v)\geq \frac{ndx}{(1-x)(1+dx)}.$$
Whence
$$-\frac{\tau'(G,x)}{\tau(G,x)}\geq -\frac{n}{d+1}\frac{\tau'(K_{d+1},x)}{\tau(K_{d+1},x)}.$$
After integration and multiplication with $-1$ we get that
$$\ln \tau(G,x)\leq \frac{n}{d+1}\ln \tau(K_{d+1},x).$$
This is equivalent with the statement.
\end{proof}

Another interesting application of Lemma~\ref{Schur complement bounds} is that vertex-transitive graphs are \\ $M$-graphs, that is, if $G$ is vertex-transitive, then for all $x\in (0,1)$ all edges are of type I.

\begin{proof}[Proof of Theorem~\ref{vertex-transitive}]
Since $G$ is vertex-transitive the vector $\underline{1}$ is an eigenvector of $\B_G(x)$, and since it is positive definite we have $\B_G(x)\underline{1}=\lambda \underline{1}$, where $\lambda>0$. Thus
$$1+(x-1)Y_u=(1+xY_u)-\sum_{v\in N_G(u)}y_G(u,v)=\lambda> 0.$$
Hence $Y_u<\frac{1}{1-x}$. Together with part (b) of  Lemma~\ref{Schur complement bounds}, that is, with
$$\frac{2}{1-x}\leq (1+xY_u)+(1+xY_v)+2y_G(u,v)$$ we get that $y_G(u,v)>0$.
Together with part (c) of Lemma~\ref{Schur complement bounds}, that is, with
$$\frac{2}{1-z_G(u,v)}\leq (1+xY_u)+(1+xY_v).$$
we get that 
$$\frac{2}{1-z_G(u,v)}<\frac{2}{1-x}.$$
Equivalently, $z_G(u,v)<x$.
\end{proof}

\begin{Rem}
We could have proved slightly stronger inequalities as Lemma~\ref{log-derivative_upper_bound} asserts that
$$2\sum_{e\in E(G)}y_G(u,v)\leq \frac{n-1}{1-x}$$
implying that $Y_u\leq \frac{n-1}{n}\cdot \frac{1}{1-x}$ for a vertex-transitive graph. This, in turn, implies that
$$y_G(u,v)\geq \frac{1}{n}\cdot \frac{x}{1-x}.$$
\end{Rem}

\section{Ihara zeta function} \label{Ihara-zeta_section}

In this section we relate the function $\tau(G,x)$ to the Ihara zeta function. This will enable us to prove Theorem~\ref{tightness} for the first interval and Theorem~\ref{large girth limit}.
\medskip

Let $I$ be the identity matrix of size $|V(G)|\times |V(G)|$. Furthermore, let $A$ be the adjacency matrix of the graph $G$ and let $D$ be the diagonal matrix consisting of the degrees of the graph $G$. Bass \cite{Bas} proved the following expression for the so-called  Ihara zeta function \cite{Ih,Has,KotSun,StaTer} of the graph $G$:
$$\zeta_G(x)=\frac{1}{(1-x^2)^{|E(G)|-|V(G)|}\det(I-xA+(D-I)x^2)}.$$
This is not the original definition, but for the sake of simplicity we will consider this expression to be the definition of the Ihara zeta function.
Let 
$$\textbf{Z}_G(x)=\frac{1}{1-x^2}\left(I-xA+(D-I)x^2\right).$$
If $|x|<\frac{1}{\Delta-1}$, then $\textbf{Z}_G(x)$ is diagonally dominant, consequently positive definite. Moreover, $\textbf{Z}_G(x)=B(\underline{t})$, where $t(u,v)=\frac{x}{1-x^2}$ for all $(u,v)\in E(G)$. Hence $\textbf{Z}_G(x)\in \mathcal{B}(G,x)$ for $|x|<\frac{1}{\Delta-1}$. This shows that
$$\det(\textbf{Z}_G(x))\leq \det(\B_G(x))=\frac{1}{\tau(G,x)},$$
and consequently,
$$\zeta_G(x)(1-x^2)^{e(G)}\geq \tau(G,x).$$

Now we are ready to prove Theorem~\ref{tightness} for the interval $[0,\frac{1}{\Delta-1}]$.

\begin{proof}[Proof of Theorem~\ref{tightness} for the first interval.]
Recall that we need to prove that if $G$ is graph with $e(G)$ edges, largest degree $\Delta$ and girth $g$, and $x\in [0,\frac{1}{\Delta-1})$, then
$$\left| \frac{\ln \tau(G,x)}{e(G)}-\ln (1-x^2)\right|\leq 2\frac{((\Delta-1)x)^g}{1-(\Delta-1)x}.$$
Since 
$$\tau(G,x)\geq (1-x^2)^{e(G)}$$
for $x\in (0,1)$ we immediately get that
$$\frac{\ln \tau(G,x)}{e(G)}\geq \ln(1-x^2).$$
We need to prove the inequality
$$ \frac{\ln \tau(G,x)}{e(G)}\leq \ln (1-x^2)+2\frac{((\Delta-1)x)^g}{1-(\Delta-1)x}.$$
We will use the fact for $|x|<\frac{1}{\Delta-1}$ we have
$$\tau(G,x)\leq \zeta_G(x)(1-x^2)^{e(G)}.$$
So we only need to prove that for $|x|<\frac{1}{\Delta-1}$ we have
$$\frac{\ln \zeta_G(x)}{e(G)}\leq 2 \frac{((\Delta-1)|x|)^g}{1-(\Delta-1)|x|}.$$
Here we use an alternative description of $\zeta_G(x)$ due to Bass \cite{Bas}.
Let us replace all edges of the graph $G$ with a pair of directed edges going opposite ways. Then we can define the directed edge matrix $M$ of size $2e(G)\times 2e(G)$ as follows: for directed edges $e$ and $f$ let $M_{ef}=1$ if the head of $e$ is the tail of $f$, and the tail of $e$ is not the head of $f$, otherwise all entries of $M$ are $0$. Then
$$\zeta_G(x)^{-1}=\det(I-xM),$$
where $I$ is the identity matrix of size $2e(G)\times 2e(G)$. Let $\rho_1,\dots ,\rho_{2e(G)}$ be the eigenvalues of $M$. These eigenvalues are not necessarily real, but $|\rho_i|\leq \Delta-1$ as every row of $M$ contains at most as many $1$'s. Hence
$$\ln \det(I-xM)=\sum_{i=1}^{2e(G)}\ln (1-x\rho_i)=\sum_{i=1}^{2e(G)}\sum_{k=1}^{\infty}\frac{-(x\rho_i)^k}{k}=\sum_{k=1}^{\infty}\frac{-1}{k}\sum_{i=1}^{2e(G)}(x\rho_i)^k.$$
Now observe that if $k\leq g-1$, then
$$\sum_{i=1}^{2e(G)}\rho_i^k=\mathrm{Tr}M^k=0.$$
Hence
$$\ln \det(I-xM)=\sum_{k=g}^{\infty}\frac{-1}{k}\sum_{i=1}^{2e(G)}(x\rho_i)^k.$$
Hence
$$\ln \zeta_G(x)\leq \sum_{k=g}^{\infty}\frac{1}{k}\sum_{i=1}^{2e(G)}|x\rho_i|^k\leq \sum_{k=g}^{\infty}\sum_{i=1}^{2e(G)}|x\rho_i|^k\leq \sum_{i=1}^{2e(G)}\frac{|x\rho_i|^g}{1-|x\rho_i|}\leq 2e(G)\frac{((\Delta-1)|x|)^g}{1-(\Delta-1)|x|}.$$
Hence for $x\in [0,\frac{1}{\Delta-1})$ we have
$$\frac{\ln \tau(G,x)}{e(G)}\leq \ln (1-x^2)+2\frac{((\Delta-1)|x|)^g}{1-(\Delta-1)|x|}.$$

\end{proof}

\begin{proof}[Proof Theorem~\ref{large girth limit}]
This is trivial from part (a) of Theorem~\ref{tightness}.
\end{proof}

\section{Spanning trees of regular graphs} \label{spanning tree}

In this section, we prove an upper bound on the number of spanning trees of regular graphs. This result is only weaker in the subexponential term than the corresponding result of B. McKay \cite{mckay}, and its proof is completely different.

\begin{Th} \label{McKay's bound} Let $\tau(G)$ be the number of spanning trees of a $d$--regular graph $G$ on $n$ vertices. Then
$$\tau(G)\leq \frac{e(d-1)}{d(d-2)}\left(\frac{(d-1)^{d-1}}{(d^2-2d)^{d/2-1}}\right)^n,$$
where $e$ is the base of the natural logarithm.
\end{Th}

\begin{Def} For a graph $G$ the $V\times V$ Laplacian matrix $L(G)$ is defined as follows:
the diagonal element $L(G)_{vv}=d_v$, the degree of the vertex $v$, and for $u\neq v$ we have
$$L(G)_{uv}=\left\{\begin{array}{cc} -1 & \mbox{if}\ (u,v)\in E(G),\\ 0 & \mbox{if}\ (u,v)\notin E(G).\end{array}\right.$$
\end{Def}

The following lemma is a simple corollary of Kirchhoff's matrix-tree theorem (see Theorem 13.2.1 and Lemma 13.2.4 of \cite{GodRoy}). 

\begin{Lemma} Let $G$ be a graph on $n$ vertices. The Laplacian matrix $L(G)$ is positive semidefinite with $0$ being the smallest eigenvalue. Furthermore, if $\lambda_1\geq \dots \geq \lambda_n=0$ are the eigenvalues of $L(G)$, then the number of spanning trees $\tau(G)$ satisfies
$$\tau(G)=\frac{1}{n}\prod_{i=1}^{n-1}\lambda_i.$$
\end{Lemma}

\begin{proof}[Proof of Theorem~\ref{McKay's bound}] Let $x=\frac{1}{d-1}$ and let $t=\frac{n-1}{nd(1-x)}$.
Let us consider the matrix $B\in \mathcal{B}(G,x)$ for which $t(u,v)=t$ for all $(u,v)\in E(G)$. Then the obtained matrix $B$ is positive definite since it is diagonally dominant:
$$(1+xdt)-dt=1+(x-1)dt=1-\frac{n-1}{n}=\frac{1}{n}.$$
From this we can see that
$$B=\frac{1}{n}I+t\cdot L(G),$$
where $L(G)$ is the Laplacian-matrix of $G$. Let $\lambda_1\geq \dots \geq \lambda_n=0$ be the eigenvalues of the Laplacian-matrix. Then
$$\det(B)=\prod_{i=1}^n\left(t\lambda_i+\frac{1}{n}\right)=\frac{1}{n}\prod_{i=1}^{n-1}\left(t\lambda_i+\frac{1}{n}\right)\geq \frac{1}{n}\prod_{i=1}^{n-1}\left(t\lambda_i\right)=t^{n-1}\tau(G).$$
In the last step we have used the formula
$$\tau(G)=\frac{1}{n}\prod_{i=1}^{n-1}\lambda_i.$$
Note that $B\in \mathcal{B}(G,x)$ whence
$$\frac{1}{\tau(G,x)}=\det \B_G(x)\geq \det B\geq t^{n-1}\tau(G).$$
Thus 
$$\tau(G)\leq \frac{1}{t^{n-1}\tau(G,x)}=d^{n-1}\left(\frac{n}{n-1}\right)^{n-1}\frac{(1-x)^{n-1}}{\tau(G,x)}\leq ed^{n-1}\frac{(1-x)^{n-1}}{\tau(G,x)}.$$
Now we use Theorem~\ref{tau-main} and $x=\frac{1}{d-1}$ to get that
$$\tau(G,x)\geq (1-x^2)^{e(G)}=\left(1-\frac{1}{(d-1)^2}\right)^{nd/2}=\frac{d-2}{d-1}\left( \frac{d^{d/2}(d-2)^{d/2-1}}{(d-1)^{d-1}}\right)^n(1-x)^{n-1}.$$
From this we obtain that
$$\tau(G)\leq \frac{e(d-1)}{d(d-2)}\left(\frac{(d-1)^{d-1}}{(d^2-2d)^{d/2-1}}\right)^n.$$
\end{proof}

\section{Multivariate case} \label{multivariate}

In this section we consider the  non-homogeneous (multivariate) version of some of our claims. Since the proofs are straightforward modifications of the homogeneous cases we do not detail the proofs.

For a graph $G$ and $\underline{x}\in [0,1]^{E(G)}$ let $\mathcal{A}(G,\underline{x})$ be the set of positive definite matrices $A$ which has diagonal elements $1$'s, and if $(u,v)\in E(G)$, then $A_{u,v}=x_{u,v}$. Let
$$\tau(G,\underline{x})=\max_{A\in \mathcal{A}(G,\underline{x})}\det A,$$
and let $\A_G(\underline{x})$ be the matrix, where the maximum is achieved (this is unique) if the set $\mathcal{A}(G,\underline{x})$ is not empty.
The entries of this matrix will be denoted by $\A_G(\underline{x})_{u,v}=z_m(u,v)$. Since $\A_G(\underline{x})\in \mathcal{A}(G,\underline{x})$ we have $z_m(u,u)=1, z_m(u,v)=x_e$ if $(u,v)=e\in E(G)$.

\begin{Th} The optimization problem has a unique maximizer $\A_G(\underline{x})$ if $\mathcal{A}(G,\underline{x})$ is not empty. If $u$ and $v$ are not adjacent vertices, then $\A_G(\underline{x})^{-1}_{u,v}=0.$
\end{Th}

We will denote the inverse of $\A_G(\underline{x})$ by $\B_G(\underline{x})$. It will be convenient to parametrize $\B_G(\underline{x})$ as follows: $\B_G(\underline{x})_{u,v}=-y_m(u,v)$ if $(u,v)\in E(G)$, $\B_G(\underline{x})_{u,u}=1+\sum_{v\in N_G(u)}z_m(u,v)y_m(u,v)$, where $z_m(u,v)=x_{u,v}$ for $e=(u,v)\in E(G)$. We have seen that $\B_G(\underline{x})_{u,v}=0$ if $(u,v)\notin E(G)$.

There are many results in this paper that naturally extend to the multivariate case. For instance, Theorem~\ref{dual} and Theorem~\ref{edge-deletion} have both multivariate counterparts.

\begin{Lemma} \label{dual_multivariate} Let $\mathcal{B}(G,\underline{x})$ be the set of positive definite matrices $B=B(\underline{t})$ which are parametrized as follows:
\begin{enumerate}
\item if $(u,v)\notin E(G)$, then $B_{u,v}=0$,
\item if $(u,v)\in E(G)$, then $B_{u,v}=-t_m(u,v)$,
\item for $u\in V(G)$ we have $B_{u,u}=1+\sum_{v\in N_G(u)}x_{u,v}t_m(u,v)$.
\end{enumerate}
Then $\det(B)$ is a strictly log-concave function on $\mathcal{B}(G,\underline{x})$, and it takes its maximum at the unique $B(\underline{y})$ for which $B(\underline{y})^{-1}=\A_G(\underline{x})$, i. e., $B(\underline{y})=\B_G(\underline{x})$.
\end{Lemma} 

For trees the following statement summarizes the basic facts.

\begin{Prop}
Let $T$ be a tree and suppose that for each edge $e$ a number $x_e\in (-1,1)$ is given. Let $\A_T(\underline{x})$ be the matrix whose $uv$ entry is $\prod_{e\in P}x_e$, where $P$ is the unique path connecting the vertices $u$ and $v$. Then $\A_T(\underline{x})$ is a positive definite matrix, its inverse $\B_T(\underline{x})$ has the following entries: the $uv$ entry is $0$ if $u$ and $v$ are distinct not adjacent vertices, $-\frac{x_e^2}{1-x_e^2}$ if $e=(u,v)$, and $1+\sum_{e: u\in e}\frac{x_e^2}{1-x_e^2}$ if $u=v$. The determinant of $\A_T(\underline{x})$ is $\prod_{e\in E(T)}(1-x_e^2)$.
\end{Prop}

\begin{Th} Suppose that for some graph $G$, an edge $e$ and some $\underline{x}\in [0,1]^{E(G)}$ we have $|y_e|\leq \frac{x_e}{1-x_e^2}$. Then
$$\tau(G,\underline{x})\geq (1-x_e^2)\tau(G-e,\underline{x}).$$
\end{Th}

\begin{Rem}
It might be tempting to believe that
$$\tau(G,\underline{x})\geq \prod_{e\in E(G)}(1-x_e^2),$$
but this is not true in general. The proof of the multivariate version of Theorem~\ref{tau-main} fails at the point that Theorem~\ref{edge-contraction} has no multivariate counterpart. 
\end{Rem}

\section{Open problems} \label{open problems}

We end this paper with some open problems.
\bigskip

The following problem is motivated by Theorem~\ref{large girth limit}, where the question is answered for the interval $x\in [0,\frac{1}{d-1})$.

\begin{?} \label{problem 1}
Let $(G_n)_n$ be a sequence of $d$--regular graphs with girth $g(G_n)\to \infty$.  Is it true that for all $x\in [0,1)$, the limit 
$$\lim_{n\to \infty}\frac{\ln \tau(G_n,x)}{v(G_n)}$$
exists, and if it exists what is it?
\end{?}

A similar question concerns with the infinite graph $\mathbb{Z}^d$.

\begin{?} \label{problem 2}
Let $(G_n)_n$ be a sequence of graphs converging to $\mathbb{Z}^d$, for instance, larger and larger boxes.  Is it true that for all $x\in [0,1)$, the limit 
$$\widetilde{\tau}(\mathbb{Z}^d,x)=\lim_{n\to \infty}\frac{\ln \tau(G_n,x)}{v(G_n)}$$
exists, and if it exists what is it?
\end{?}

\begin{?} \label{problem 3}
Is it true that the function $\frac{\tau(G,x)}{(1-x^2)^{e(G)}}$ is monotone increasing? Equivalently,
$$-\frac{\tau'(G,x)}{\tau(G,x)}\leq 2e(G)\frac{x}{1-x^2}.$$
\end{?}

\begin{?} \label{problem 4}
Is it true that if $x\geq 0$, then all elements of the matrix $\A_G(x)$ are non-negative? Is it true that if $G$ is connected and $x>0$, then all elements of the matrix $\A_G(x)$ are positive?
\end{?}

\begin{?} \label{problem 5}
Is it true that for all graph $G$ and edge $(u,v)\in E(G)$, and fixed $x\geq 0$ we have $y_G(u,v)\leq \frac{x}{1-x^2}$?
\end{?}

Naturally, Problem~\ref{problem 5} implies Problem~\ref{problem 3}. We also believe that Problem~\ref{problem 4} and Problem~\ref{problem 5} are actually equivalent. 

\begin{?}
Is it true that if $G$ is a random Erd\H os-R\'enyi graph or a random regular graph, then with high probability $y_G(u,v)\geq 0$ for all $(u,v)\in E(G)$ and $x\in (0,1)$? 
\end{?}

\section{Appendix: tools from probability theory and matrix analysis} \label{preliminaries}

In this section we collected a few results from probability theory and matrix analysis that we use in the paper.

\subsection{Probability theory}

Let $Z_1,\dots ,Z_k$ be independent random variables with standard normal distributions. Then
the random variable $X=\left(\sum_{i=1}^kZ_i^2\right)^{1/2}$ has a chi distribution with parameter $k$ denoted by $\chi_k$, see chapter 11 of the book \cite{FEHP}. Its probability density function is  
$$f_k(x)=\left\{ \begin{array}{ll} \frac{1}{2^{k/2-1}\Gamma\left(\frac{k}{2}\right)}x^{k-1}e^{-x^2/2} & \mbox{if} \ x\geq 0, \\
0 & \mbox{otherwise.}\end{array} \right.$$
Here $\Gamma(z)=\int_0^{\infty}x^{z-1}e^{-x}dx$ for $z>0$. 
The random variable $X^2=\sum_{i=1}^kZ_i^2$ has a chi-square distribution with parameter $k$ denoted by $\chi_k^2$. Its probability density function is
$$g_k(x)=\left\{ \begin{array}{ll} \frac{1}{2^{k/2}\Gamma\left(\frac{k}{2}\right)}x^{k/2-1}e^{-x/2} & \mbox{if} \ x\geq 0, \\
0 & \mbox{otherwise.}\end{array} \right.$$

An $\textbf{X}=(X_1,\dots ,X_n)$ has a non-degenerate multivariate normal distribution if for the 
vector $\mu=(\E X_1,\dots ,\E X_n)$ and positive definite covariance matrix $\Sigma$ 
with entries $\Sigma_{ij}=\E(X_iX_j)-\E X_i\cdot \E X_j$ we have the the density function
$$f_{\textbf{X}}(x)=\frac{1}{(2\pi)^{n/2}\det(\Sigma)^{1/2}}
\exp \left(-\frac{1}{2}(x-\mu)^T\Sigma^{-1}(x-\mu)\right).$$

Let $G$ be a $n\times k$ matrix such that each column vector is independently chosen from an $n$-variate normal distribution with zero mean and covariance matrix $\Sigma$. Then the random matrix $GG^T$ is an $n\times n$ positive semidefinite random matrix with $k$ degree of freedom. Note that when $n=1$ and $\Sigma=I$ we get back the chi-square distribution. In general, this is the Wishart distribution \cite{Wish}, its probability density function is the following:
$$f(M)=\frac{1}{2^{nk/2}\det(\Sigma)^{k/2}\Gamma_{n}\left(\frac{k}{2}\right)}\det(M)^{(k-n-1)/2}e^{-\frac{1}{2}\mathrm{Tr}(\Sigma^{-1}M)}$$
if $M$ is a positive definite matrix, and $0$ otherwise. Here $\Gamma_n(z)$ is the multivariate Gamma function defined as follows:
$$\Gamma_n(z)=\pi^{n(n-1)/4}\prod_{j=1}^n\Gamma(z+(1-j)/2).$$

\subsection{Matrix analysis}

In this section we collected a few results from linear algebra that we use in the paper. All of them can be found in the book \cite{HJ}.
\bigskip

The following statement can be found in \cite{HJ}, section 0.7.3 and 0.8.4 and section 7.7 (with an emphasis on equation 7.7.5).

\begin{Th} \label{Schur-complement}
Let $A=\left(\begin{array}{cc} A_{11} & A_{12} \\ A_{21} & A_{22} \end{array}\right)$ be a block matrix and  let $B=\left(\begin{array}{cc} B_{11} & B_{12} \\ B_{21} & B_{22} \end{array}\right)$ be the corresponding decomposition of the inverse matrix. Then
$$\left(\begin{array}{cc} A_{11} & A_{12} \\ A_{21} & A_{22} \end{array}\right)=\left(\begin{array}{cc} (B_{11}-B_{12}B_{22}^{-1}B_{21})^{-1} & -B_{11}^{-1}B_{12}(B_{22}-B_{21}B_{11}^{-1}B_{12})^{-1} \\
-B_{22}^{-1}B_{21}(B_{11}-B_{12}B_{22}^{-1}B_{21})^{-1} & (B_{22}-B_{21}B_{11}^{-1}B_{12})^{-1} \end{array}\right).$$
supposing that the appropriate matrices are invertible.
Furthermore,
$$\det(B)=\det(B_{22})\det(B_{11}-B_{12}B_{22}^{-1}B_{21})=\det(B_{22})\det(A_{11})^{-1}.$$
Furthermore, the matrix $A$ is positive definite if and only if $A_{11}$ and $A_{22}-A_{21}A_{11}^{-1}A_{12}$ are positive definite.
\end{Th}

The following theorem is Theorem 7.2.5 in \cite{HJ}.

\begin{Th}[Sylvester's criterion] \label{Sylvester}
Let $A$ be a symmetric matrix of size $n\times n$. For $1\leq k\leq n$ let $A_k$ denote the matrix induced by the first $k$ rows and $k$ columns. 
Suppose that $\det(A_k)>0$ for all $1\leq k\leq n$. Then $A$ is positive definite.
\end{Th}

The following theorem is Theorem 7.8.16 in \cite{HJ}.

\begin{Th}[Oppenheim's inequality] \label{Opponheim}
Let $A$ and $B$ be two positive definite matrices of size $n\times n$. Let $C$ be their Hadamard-product: $C_{ij}=A_{ij}B_{ij}$. Then
$$\det(C)\geq \left(\prod_{i=1}^na_{ii}\right)\det(B).$$
\end{Th}

For the following theorem, see Theorem 7.6.6 and Corollary 7.6.8 in \cite{HJ}. 

\begin{Th} \label{log-concave-pos-def}
Let $A$ and $B$ positive definite matrices and $\alpha\in (0,1)$, then 
$$\det(\alpha A+(1-\alpha)B)\geq \det(A)^{\alpha}\det(B)^{1-\alpha}.$$
\end{Th}

For the following theorem, see Section 7.4.12 in \cite{HJ}.

\begin{Th} \label{quadratic-pos-def}
Let $A$ be a positive definite matrix of size $n\times n$ and $\underline{x}\in \mathbb{R}^n$. Then
$$\underline{x}^TA\underline{x}\cdot \underline{x}^TA^{-1}\underline{x}\geq ||\underline{x}||_2^4.$$
\end{Th}

For the following theorem, see 8.3P15 in \cite{HJ}, or for a more comprehensive treatment see  Chapter 2.5 in \cite{HJ2}.

\begin{Th} \label{M-matrix-inverse}
Suppose that $B$ is a  positive definite matrix for which $B_{ij}\leq 0$ whenever $i\neq j$. Then all elements of $B^{-1}$ are non-negative.
\end{Th}

\textbf{Acknowledgment.} The authors are very grateful to L\'aszl\'o Lov\'asz for useful discussions. They also thank the anonymous referee their comments concerning the presentation of the paper.

\end{document}